\documentclass[10pt]{article}
\input epsf.tex
\usepackage[T1]{fontenc}
\usepackage{ae}
\usepackage{amssymb,amsmath,amsthm,textcomp,amsxtra}
\usepackage[mathcal]{euscript}
\usepackage{graphicx,pst-plot,mparhack}
\usepackage[abs]{overpic}
\usepackage[matrix,arrow]{xy}
  \xymatrixrowsep{0.3pc}
  \xymatrixcolsep{0.3pc}
\usepackage{enumerate,makeidx}

\renewcommand*{\geq}{\geqslant}
\renewcommand*{\leq}{\leqslant}

\makeindex

\numberwithin{equation}{section}

\newcommand*{\abs}[1]{\lvert#1\rvert}

\newcommand*{\Sscr}{\mathcal S}

\newcommand*{\N}{\mathbb{N}}
\newcommand*{\R}{\mathbb{R}}

\renewcommand*{\d}{\mathrm{d}}

\DeclareMathOperator{\var}{var}

\usepackage{amsfonts,amsmath,latexsym,amsthm,amssymb}
\usepackage{dsfont} 
\usepackage{mathtools}
\usepackage{enumerate}
\usepackage{mathrsfs} 
\mathtoolsset{showonlyrefs}
\usepackage{listings} 
\usepackage{amscd}
\DeclareMathOperator{\const}{const}

\def\wt{\widetilde}
\def\pas{\ \d P \mbox{-a.s.}}
\def\<{\left<}
\def\>{\right>}
\def\({\left(}
\def\){\right)}
\def\vp{\varphi}
\def\o{\omega}
\def\O{\Omega}
\def\9{\infty}
\def\R{\mathbb R}
\def\N{\mathbb N}

\def\sha{{\cal A}}
\def\shb{{\cal B}}
\def\shc{{\cal C}}

\def\she{{\cal E}}
\def\shf{{\cal F}}
\def\shg{{\cal G}}
\def\shh{{\cal H}}
\def\shm{{\cal M}}

\def\shs{{\cal S}}
\def\shy{{\cal Y}}
\def\shw{{\cal W}}

\newcommand{\norm}[1]{\left\| #1 \right\|}

\newtheorem{theo}{Theorem}[section]
\newtheorem{lemma}[theo]{Lemma}
\newtheorem{Assumption}[theo]{Assumption}

\newtheorem{prop}[theo]{Proposition}
\newtheorem{rem}[theo]{Remark}
\newtheorem{cor}[theo]{Corollary}
\newtheorem{defi}[theo]{Definition}

\newcommand{\beqnar}{\begin{eqnarray*}}
\newcommand{\eeqnar}{\end{eqnarray*}}
\newcommand{\ba}{\begin{array}}
\newcommand{\ea}{\end{array}}

\newcommand{\halb}{\frac{1}{2}}

\newenvironment{prooff}[1]{\begin{trivlist}
\item {\it \bf Proof}\quad} 
{\qed\end{trivlist}}

\begin{document}

\title{
Doubly probabilistic representation for the
stochastic  porous media type equation.}

\author{ Viorel Barbu (1), Michael R\"ockner (2)
and Francesco Russo (3) } 

\date{August 3rd  2016}
\maketitle

\thispagestyle{myheadings}
\markright{Stochastic porous media with multiplicative noise.}

{\bf Summary:} The purpose of the present paper consists in proposing 
and discussing a doubly probabilistic representation
for a stochastic porous media equation in the whole space $\R^1$
perturbed by a multiplicative colored noise.
For almost all random realizations $\omega$, one associates
a stochastic differential equation in law with random coefficients,
driven by an independent Brownian motion.

{\bf Key words}: stochastic partial differential equations;
infinite volume; singular  porous media type equation;
doubly probabilistic representation; multiplicative noise;
singular random Fokker-Planck type equation; filtering.

{\bf2000  AMS-classification}: 35R60; 60H15; 60H30; 60H10; 60G46;
 35C99; 58J65; 82C31.



\begin{itemize}
\item[(1)] Viorel Barbu,
University Al.I. Cuza, Ro--6600 Iasi,
Romania.
\item[(2)] Michael R\"ockner,
Fakult\"at f\"ur Mathematik, 
Universit\"at   Bielefeld, 
\\ D--33615 Bielefeld, Germany
\item[(3)] Francesco Russo,
ENSTA ParisTech, Universit\'e Paris-Saclay,
Unit\'e de Math\'ematiques appliqu\'ees,
 828, boulevard des Mar\'echaux,
F-91120 Palaiseau, France.

\end{itemize}

\vfill \eject

\section{Introduction}

\setcounter{equation}{0}


We consider a function $\psi: \R \rightarrow \R$ and
real functions $e^0, \ldots, e^N$ on $\R$, for some strictly positive integer 
$N$. In the whole paper, the following assumption will be in force.
\begin{Assumption}\label{E1.0}
\begin{itemize}
\item 
$\vert \psi (u) \vert \le  {\rm const} \vert u \vert, \ u \ge 0.  $ 
In particular,  
$\psi(0) = 0$.
\item  $\psi: \R \rightarrow \R$ is a continuous function such that
its restriction to ${\R_+}$
is monotone increasing.
Moreover we also suppose that
$\lim_{u \rightarrow 0} \frac{\psi(u)}{u}$ exists.
\item Let  $e^i \in  C^2_{\rm b}(\R), 0 \le i \le N$.
\end{itemize}
\end{Assumption}
 Let $ T > 0$ and $(\Omega, \shf, P)$,
be a fixed  probability space.
A generic element of $\Omega$ will be denoted by $\omega$.
 $(\shf_t, t \in [0,T])$ will stand for a filtration,
 fulfilling the usual conditions and 
we suppose $\shf = \shf_T$.
Let $\mu(t,\xi), t \in [0,T], \xi \in \R,$ be a random field
of the type
$$ \mu(t,\xi) = \sum_{i=1}^N e^i (\xi) W^i_t + e^0(\xi) t,
 \ t \in [0,T], 
\xi \in \R,$$
where $W^i, 1 \le i \le N,$ are independent continuous
 $(\shf_t)$-Brownian motions   on $(\Omega, \shf, P)$, 
 which are fixed from now on until the end of the paper.

For technical reasons we will sometimes set
$W^0_t \equiv t$.
We focus on a stochastic partial differential equation 
of the following type:
\begin{equation}
\label{PME}
\left \{
\begin{array}{ccl}
\partial_t X(t,\xi)&=& \frac{1}{2} \partial_{\xi \xi}^2(\psi(X(t, \xi) ) +
 X(t,\xi) \partial_t \mu(t, \xi),
\\
X(0,\d \xi)& = & x_0(\d \xi), 
\end{array}
\right.
\end{equation}
which holds in the sense of Definition \ref{DSPDE}, where $x_0$ is a
a given probability measure on $\R$.
The stochastic multiplication above is of It\^o type.
We look for a solution of (\ref{PME}) with time evolution in $L^1(\R)$.
Since $\psi$ restricted to $\R_+ $ is non-negative,
 Assumption \ref{E1.0} implies
   $ \psi (u) = \Phi^2(u) u, \ u  \ge 0$, $\Phi: \R_+ \rightarrow \R$
   being a
non-negative continuous function which is  bounded on $\R_+$. 
\begin{rem}\label{Rint} 
\begin{enumerate}
\item In the sequel we will consider, without further comments
extensions of $\psi$ (and $\Phi$) to the real line which fulfill
the first two items of Assumption \ref{E1.0} for $u \in \R$ instead of 
$u \ge 0$. 
\item The restriction on $u \mapsto \Phi(u)$ introduced in Assumption \ref{E1.0}
 to be continuous is not 
always necessary, but here
we assume this for simplicity.
\end{enumerate} 
\end{rem}

When $\psi (u) = \vert u \vert^{m-1}u$,
 $m  >   1$, (\ref{PME}) and $\mu \equiv 0$, 
\eqref{PME}
 is  nothing else but the
classical {\it porous media equation}.
When $\psi $ is a general increasing function
(and $\mu \equiv 0$),
there are several contributions to 
the analytical study of (\ref{PME}),
 starting from
\cite{BeBrC75} for existence,   \cite{BrC79} for  uniqueness in the
case of bounded solutions 
 and \cite{BeC81} 
for continuous dependence on the coefficients.
Those are the classical references when the space variable varies
on the real line. For equations in a bounded domain and Dirichlet
boundary conditions, for simplicity, we only refer to monographs,
e.g. \cite{vazquez, Show, barbu93, barbu10}.

As far as the stochastic porous media is concerned,
most of the work for existence and uniqueness
concerned the case of bounded domain, see for instance
\cite{BDR09, BDR09CMP, BDR08}.
 In the infinite volume case, i.e. when the underlying domain is $\R^d$, 
well-posedness  was fully analyzed in  \cite{Ren}, when
 $\psi$ is polynomially bounded
 (including the fast diffusion case) 
when the space dimension is $d \ge 3$.
\cite{BRR3} established existence and uniqueness for any dimension
$d \ge 1$ and the authors obtained estimates for finite time extinction.
To the best of  our knowledge, except for  \cite{Ren} 
and   \cite{BRR3}, 
 this seems to be  the only work concerning a stochastic porous type equation
in infinite volume.

  We provide a  probabilistic representation of solutions to \eqref{PME}
extending the results of \cite{BRR1, BRR2} which treated the 
deterministic case $\mu \equiv 0$.
In the deterministic case,
it seems that the first author who considered a
probabilistic representation (of the type studied in this paper) for the
solutions of a non-linear deterministic PDE was McKean
\cite{mckean}, particularly in relation with the so called propagation of
chaos. In his case, however, the coefficients were smooth. From then on
the literature steadily grew and nowadays there is a vast amount of
contributions to the subject, see the reference 
list of \cite{BRR1, BRR2}.
A probabilistic representation
 when  $\psi(u) = \vert u \vert  u^{m-1}, m  > 1,$ was provided for
instance in \cite{BCRV}, in
 the case of the classical 
porous media equation. When $m < 1$, i.e. in the case of the fast
diffusion equation, \cite{BR} provides a probabilistic representation of the
so called {\bf Barenblatt solution}, i.e. the solution whose initial condition 
is concentrated at zero.
 
\cite{BRR1, BRR2} discussed the probabilistic representation
when $\mu = 0$  in the so called non-degenerate and degenerate case respectively (see Definition  
\ref{DNond}),
  where $\psi$ also may have jumps. 

In the case $\mu =  0$, the equation \eqref{PME}
models a non-linear phenomenon macroscopically.
Let us denote by $u:[0,T] \times \R \rightarrow \R$ the
solution of that equation.
The idea of the probabilistic representation is
to find a process $(Y_t, t \in [0,T])$ whose law 
at time $t$ has $u(t,\cdot)$ for its density.
In this case the equation \eqref{PME} is 
conservative, in the sense that the integral (mass)
of the solution is conserved along the time.

The process $Y$ turns out to be the weak solution of the 
non-linear stochastic differential equation
\begin{equation}
\label{E1.2}
\left \{
\begin{array}{ccc}
Y_t &=& Y_0 + \int_0^t \Phi(u(s,Y_s)) dB_s,  \\
{\rm Law } (Y_t) &=& u(t,\cdot), \quad t \ge 0, \\
\end{array}
\right.
\end{equation}
where $B$ is a classical Brownian motion.
The behavior of $Y$ is the microscopic counterpart of
the phenomenon described by (\ref{PME}), describing
the evolution of  a single particle, whose law 
behaves according to  (\ref{PME}).




The idea of this paper is to consider the case when $\mu \neq 0$.
 This includes the case when 
 $\mu$ is not vanishing but it is deterministic; it happens
when only $e^0$ is non-zero, and $e^i \equiv 0, 1 \le i \le n$.
 In this case our technique gives a sort of forward
Feynman-Kac formula for a non-linear PDE.
One of the main interests of this paper is that it provides a 
(forward) probabilistic representation for {\it non conservative} 
(random) PDE.

We introduce a doubly stochastic representation 
on which one can represent the solution of \eqref{PME} 
as the weighted-law with respect to the random field $\mu$
(or simply the $\mu$-weighted law) of
a solution to a non-linear SDE.

Intuitively, it describes the microscopic aspect of the SPDE \eqref{PME}
for almost all quenched $\omega$. 
The terminology strongly refers to the case where the 
probability space $(\Omega, \shf, P)$ on which the SPDE
is defined,  remains fixed.

We represent a solution $X$ to  \eqref{PME} making use of
another independent source of randomness described by another
probability space based on some set $\Omega_1$. 

The analog of the process $Y$, obtained 
when $\mu$ is zero in \cite{BRR2, BRR1},
is a doubly stochastic process, still denoted
by $Y$ defined on $(\Omega_1 \times \Omega, Q)$,
for which, $X$ constitutes the so-called
family of {\it  $\mu$-marginal weighted laws of $Y$}, see Definition 
\ref{mylau}. 
$Y$ is the solution of a {\it doubly stochastic non-linear 
diffusion} problem, see Definition \ref{DDoubleStoch}.
It will be  a (doubly) stochastic process 
$(\omega_1,\omega) \mapsto Y(\omega_1,\omega)$ solution
of 
\begin{equation}
\label{E1.2bis}
Y_t =  Y_0 + \int_0^t \Phi(X(s,Y_s,\omega)) dB_s,
\end{equation}
and $B(\cdot,\omega)$ is a Brownian motion on $\Omega_1$ for 
almost any fixed $ \omega \in \Omega$.
The solution of \eqref{E1.2bis} is in the following sense: fixing 
a  realization $\omega \in
\Omega$, $Y(\cdot,\omega)$ is a weak solution to the first line of 
\eqref{E1.2} with $u(t,\xi)= X(t,\xi,\omega)   $.
Moreover  $ X(t,\xi,\omega)  $ is the $\mu$-marginal weighted law 
of $Y_t(\cdot,\omega)$.


The paper  includes the following main achievements.
\begin{enumerate}
\item  If we replace in \eqref{E1.2bis}
$ a(s,\xi,\omega) = \Phi(X(s,\xi,\omega))$ and $a$ is bounded 
and non-degenerate, we show existence and uniqueness of    the solution,
 strongly in $\omega$, weakly in $\omega_1 \in \Omega_1$, see Proposition 
\ref{P4.1}. We also show the existence of law densities,
for $P$-almost all quenched $\omega$, see Proposition \ref{P4.2}.
\item Theorem \ref{T32} states that the $\mu$-marginal weighted laws $X$
of a solution $Y$ of a
  {\it doubly stochastic non-linear 
diffusion} problem constitute a solution of the stochastic porous media 
equation \eqref{PME}.
\item Conversely, given a solution $X$ of \eqref{PME}, under suitable
 conditions, there is  a solution $Y$ of the  doubly stochastic non-linear 
diffusion. This is discussed in  Theorem \ref{thm6.1} 
and in Theorem \ref{T73}, distinguishing respectively the cases when 
 $\psi$ is non-degenerate and degenerate, see Definition \ref{DNond}.
\item When $\psi$ is non-degenerate, then the doubly stochastic non-linear 
diffusion problem also admits uniqueness, see Theorem \ref{thm6.1}.
\item Section  \ref{SFiltering} illustrates a filtering interpretation 
for a solution
of SPDE \eqref{PME}.
    Indeed, the  $\mu$-marginal weighted laws  $X$ of a solution $Y$ 
 of a doubly stochastic non-linear 
diffusion problem \eqref{E1.2bis}  can be seen as  
 {\it conditional densities} of  $Y_t, t \in [0,T]$  
with respect to some  probability measure.
\item Uniqueness of the stochastic Fokker-Planck equation
obtained replacing $\Phi^2$ by a 
function $a(t,\omega,\xi)$ in \eqref{PME}, see Theorem \ref{T51}. 
\item Existence of a density to the solution of \eqref{E1.2bis},
see Proposition \ref{P4.2}.
\end{enumerate}


\section{Preliminaries} 

\setcounter{equation}{0}

\label{S2}

\subsection{Basic notations}

First we introduce some basic recurrent notations.
$\shm(\R)$ 
denotes the space of
finite real 
 measures. 
 \\
We recall that  $\shs(\R)$ is the space of the Schwartz fast decreasing
test functions.  $\shs'(\R)$ is its dual, i.e. 
the space of Schwartz tempered distributions.
On $\shs'(\R)$, the map $(I-\Delta)^\frac{s}{2},  s \in \R,$ is well-defined.
For $s \in \R$, $H^s(\R)$ denotes the classical
Sobolev space consisting of all functions  $f \in \shs'(\R)$ such that 
$(I-\Delta)^\frac{s}{2} f \in L^2(\R)$.
We introduce   the norm
 $$\Vert f \Vert_{H^s} := \Vert (I-\Delta)^\frac{s}{2} f \Vert_{L^2},$$
where $\Vert \cdot \Vert_{L^p}$ is the classical $L^p(\R)$-norm
for $ 1 \le p \le \infty$.
In the sequel, we will often simply denote $H^{-1}(\R)$, 
by  $H^{-1}$ and $L^2(\R)$  by $L^2$.
Furthermore, $W^{r,p}$ denote the classical Sobolev space of order $r \in \N$
in $L^p(\R)$ for $1 \le p \le \infty$. 
\begin{defi} \label{DMultipl}
Given a function $e$ belonging to $L^1_{\rm loc}(\R) \cap  \shs'(\R)$, we say 
that it is an
{\bf $H^{-1}$-multiplier}, if the map
$ \varphi \mapsto \varphi  e$ 
is continuous from $\shs(\R)$ to $H^{-1}$ 
with respect to the $H^{-1}$-topology on both spaces.
\end{defi}
In the following lines we give some other sufficient  conditions
on a function $e$ to be an {\bf $H^{-1}$-multiplier}.

\begin{lemma} \label{LMultipl}
Let $e\, : \, \R \to \R$.  If  $e\in W^{1,\infty}$ 
(for instance  if $e \in W^{2,1}$), 
 then $e$ is a $H^{-1}(\R)$-multiplier.
In particular the functions $e^i, 0 \le i \le N$ of Definition \ref{E1.0}
are   $H^{-1}(\R)$-multipliers.
\end{lemma}

\begin{proof}
By duality arguments, we observe that it is enough to show the existence of
a constant $\shc(e)$ such that
\begin{equation}\label{Lmult1} 
\norm{eg}_{H^1} \leq \shc(e)\norm{g}_{H^1},\; \forall\; g \in \Sscr (\R).
\end{equation}
\eqref{Lmult1} follows by product derivation rules, with 
for instance $\shc(e)=\sqrt 2 \left(\norm{e}^2_\infty + 
\norm{e'}^2_\infty\right)^{\frac12}.$
\end{proof}
With respect to  the random field $\mu$, we introduce a notation 
for the It\^o type  stochastic integral below.

Let  $ Z = (Z(s,\xi), s \in [0,T], \xi \in \R)$ be a random field
on $(\Omega, \shf, (\shf_t), P) $
such that $\int_0^T \left (\int_{\R}  \vert Z(s,\xi) \vert \d \xi \right)^2 
\d s < \infty$
a.s. and it is an $L^1(\R)$-valued $(\shf_s)$-progressively measurable process.
 Then,  the stochastic integral
\begin{equation} \label{DSI}
\int_{[0,t]\times \R} Z(s, \xi)  \mu(\d s, \xi) d \xi := \sum_{i=0}^N \int_0^t 
 \left(\int_\R Z(s,\xi) e^i(\xi) \d \xi\right) \d W^i_s,
 \end{equation}
is well-defined. \\More generally, if  
$ s \mapsto Z(s,\cdot)$ is a measurable map 
$[0,T] \times \Omega \mapsto \shm(\R)$, where $\shm(\R)$ is the space of
signed finite measures, such that
$\int_0^T \Vert Z(s,\cdot) \Vert_{\rm var}^2 \d s < \infty$, 
then the stochastic integral 
\begin{equation} \label{DSIMeas}
\int_{[0,t]\times \R} Z(s, \xi)  \mu(\d s, \xi) d \xi := \sum_{i=0}^N \int_0^t 
 \left(\int_\R Z(s,\d \xi) \right) e^i(\xi) \d W^i_s,
 \end{equation}
is well-defined.

We specify now better the filtration $ (\shf_t)_ {t \in [0,T]}$ 
of the introduction.
We will consider a fixed filtered
 probability space 
$(\Omega, \shf, P, (\shf_t)_ {t \in [0,T]})$,
where  $ (\shf_t)_ {t \in [0,T]}$ is the canonical filtration
of a standard Brownian motion  $(W^1, \ldots, W^N)$ 
enlarged with the $\sigma$-field generated by $x_0$. 
   We also suppose that $\shf_0$ contains the $P$-null sets and 
$\shf = \shf_T$. 

Let $(\Omega_1, \shh)$  be a measurable space.
In the sequel, we will also consider 
 another filtered probability space
$(\Omega_0, \shg, {\bf Q}, (\shg_t)_{t \in [0,T]})$,
where $\Omega_0 = \Omega_1 \times \Omega $, $\shg = \shh \otimes \shf.$

Clearly any random element $Z$ on  $(\Omega, \shf)$
will be implicitly extended to $(\Omega_0, \shg)$
setting $Z(\omega_1,\omega) = Z(\omega)$. 
In particular 
 $W^i, i = 1 \ldots N$ will be extended in that way.

Here we fix some conventions concerning measurability.
Any  topological space $E$ is naturally equipped with its
Borel $\sigma$-algebra $\shb(E)$. 
For instance
$\shb(\R)$ (resp. $\shb([0,T]$) denotes 
the Borel $\sigma$-algebra
of $\R$ (resp. $[0,T]$).

Given any probability space $(\tilde \Omega, \tilde \shf, \tilde P)$,
the $\sigma$-field $\shf$ will always be omitted.
When we will say that a map $T: \Omega \times E \rightarrow \R$
is measurable, we will implicitly suppose that 
the corresponding $\sigma$-algebras are $\shf \otimes \shb(E)$ and $\shb(\R)$.

All the processes on any generic measurable space $(\Omega_2, \shf_2)$
will be considered to be measurable with respect to both
 variables $(t,\omega)$.
In particular any processes on  $\Omega_1 \times \Omega$
is supposed to be measurable with respect to 
$([0,T] \times \Omega_1 \times \Omega, \shb([0,T]) \otimes \shh \otimes \shf)$.

A function $(A, \omega) \mapsto Q(A, \omega)$ 
from $\shh \times \Omega \rightarrow \R_+$ 
is called {\bf random  kernel} (resp.  {\bf random probability kernel})
if for each $\omega \in \Omega$, $Q(\cdot, \omega)$
is a finite  positive (resp. probability) measure and for each $A \in \shh$,
$\omega \mapsto Q(A, \omega)$ is $\shf$-measurable.
The finite measure $Q(\cdot, \omega)$ will also be denoted
by  $Q^\omega$. 
To that random  kernel 
we can associate a specific  finite measure (resp. probability)
denoted by ${\bf Q}$ on $(\Omega_0, \shg)$
setting ${\bf Q}(A \times F) = \int_F Q(A, \omega) P(\d \omega) = \int_F Q^\omega(A) P(\d \omega) $,
for $A \in \shh, F\in \shf$.
The probability $Q$ from above 
will be supposed here and below to be associated with a random probability  kernel.

\begin{defi}\label{DSEPS}
If there is a measurable space $(\Omega_1, \shh)$
and a random kernel $Q$ as before, then the probability space
$(\Omega_0, \shg, {\bf Q})$ will be called {\bf suitable enlarged probability
space} (of $(\Omega,\shf,P)$).
\end{defi}
As said above, any random variable on $(\Omega, \shf)$ will be considered as
a random variable on $\Omega_0 = \Omega_1 \times \Omega$. Then, obviously, $W^1, \ldots, W^N$ 
are independent Brownian motions also 
$(\Omega_0, \shg, Q)$.

Given a local martingale $M$ on any filtered probability space,
 the process $Z:=\she(M)$
denotes its Dol\'eans exponential, which is a local martingale.
 In particular it is the unique
 solution of $\ dZ_t = Z_{t-} dM_t, \quad Z_0 = 1$. 
When $M$ is continuous we have 
 $Z_t = e^{M_t- \halb \langle M \rangle_t}$.

\subsection{The concept of marginal weighted  laws}
\label{nulaw}

Let us consider a suitably enlarged probability space as
in Definition \ref{DSEPS}.
\begin{defi}\label{mylau}
Let  $Y: \Omega_1 \times \Omega \times [0,T] \rightarrow \R$ 
be a measurable process, progressively measurable on $(\Omega_0, \shg, {\bf Q},
 (\shg_t)),$ where $(\shg_t)$ is some filtration on $(\Omega_0,\shg, {\bf Q})$
such that $W^1, \ldots, W^N$ are $(\shg_t)$-Brownian motions on 
$(\Omega_0,\shg,  {\bf Q})$.
We will make use of the  stochastic integral notation 
\begin{equation} \label{1.2bis}
\int_0^t \mu(ds, Y_s) = \sum_{i=0}^N
 \int_0^t e^i(Y_s) dW^i_s, t  \in [0,T].
\end{equation}
As we shall see below in Proposition \ref{P28},
 for every $t\in [0,T]$
\begin{equation} \label{e2.1}
E^{\bf Q} \left( \mathcal{E}_t \left(\int^\cdot_0 \mu (\d s, Y_s) \right) \right) < \infty.
\end{equation}
To $Y$, we will associate its {\bf family of $\mu$-marginal weighted laws},
 (or simply  {\bf family of $\mu$-weighted laws})  i.e.
 the family of random kernels ($t \in [0,T]$),
\begin{equation*}
\Gamma_t = \left(\Gamma_t^Y (A,\omega), \; A\in \shb(\R), \; 
\omega \in \Omega\right)
\end{equation*}
defined by
\begin{equation}
\label{1.2ter}
\varphi   \mapsto E^{Q^\omega} \left( \varphi (Y_t (\cdot,\omega))
 \she_t ( \int^\cdot_0 \mu (\d s, Y_s) (\cdot , \omega)) \right) 
= \int_\R \varphi(r) \Gamma_t^Y(\d r, \omega),
\end{equation}
where $\varphi$ is a generic bounded real Borel function. We will also say that for fixed $t \in [0,T], \; \Gamma_t$ is {\bf the $\mu$-marginal weighted law}
 of $Y_t$.
\end{defi}
\begin{rem} \label{R27}
\begin{enumerate}[i)]
\item If $\O$ is a singleton $\{\o_0\}, \; e^i=0,\; 1\leq i\leq N$, the
 $\mu$-marginal weighted laws coincide with the weighted  laws
\begin{equation*}
\vp \mapsto E^{\bf Q} \( \vp (Y_t) \exp \(\int^t_0 e^0 (Y_s) \d s \) \),
\end{equation*}
with ${\bf Q} =Q^{\o_0}$. In particular if $\mu \equiv0$ then the $\mu$-marginal
weighted laws are the classical laws.
\item By \eqref{e2.1}, for any $t\in [0,T]$ , for $P$ almost all
 $\o \in \Omega$,
\begin{equation} \label{ER27}
E^{Q^\o} \( \she_t (\int^\cdot_0 \mu (\d s, Y_s) (\cdot \; , \o))\)< \infty.
\end{equation}
\item The function $(t,\o)\mapsto \Gamma_t(A,\o)$ is measurable, for any $A \in
 \shb(\R)$, because $Y$ is a measurable process.
\item In the case $e^0 = 0$, the situation is the following.
 For each fixed $\omega \in \Omega$,
\eqref{1.2ter}  is a (random) non-negative measure
which is not a  probability. 
 However the expectation of its 
 total mass is indeed $1$. 
\end{enumerate}
\end{rem}
\begin{prop} \label{P28} 
Consider the situation of Definition \ref{mylau}.
Then we have the following.
\begin{enumerate}[i)]
\item 
The process
 $M_t:=\she_t \( \sum_{i=1}^N \int^\cdot_0 e^i (Y_s)\d W_s^i \)$ is a 
 martingale.
We emphasize that the sum starts indeed at $i=1$.
\item The quantity \eqref{e2.1} is bounded by
$ \exp \(T\norm{e^0}_\9 \).$
\item $E^{\bf Q} (M_t^2) \le \exp(3T \sum_{i=1}^N \Vert e^i \Vert_\infty^2),
t \in [0,T]$. Consequently $M$ is a uniformly integrable martingale.  
\item For  $P$-a.e. $\omega \in \Omega$,
$\sup_{0 \le t \le T} \norm{\Gamma_t (\cdot,\o)}_{\var} <\9,$
where we remind that $\Vert \cdot \Vert_{\var} $ stands for the total variation. 
\end{enumerate}
\end{prop}
\begin{rem} \label{R29}
Proposition \ref{P28} ii) yields in particular that
 $Y$ always admits  $\mu$-marginal weighted laws.
\end{rem}
\begin{proof}
\begin{enumerate}[i)]
\item The result follows  since the  Novikov condition
$ E^{\bf Q}  \(\exp\(\frac12 \sum_{i=1}^N \int^t_s e^i (Y_s)^2 \d s\)\)<\9$
is verified,   because the functions $e^i, \; i=1\dots N$, are bounded.
\item This follows because $E^{\bf Q} (M_t)=1, \forall t\in [0,T]$.
\item $M^2_t$ is equal to
  $N_t \exp\left(3\sum_{i=1}^N \int_0^t (e^i)^2(Y_s) \d s\right), $
where $N$ is a positive martingale with  $N_0 = 1$.
\item For $t\in [0,T]$,
\begin{align*}
\sup_{t\leq T} \norm{\Gamma_t(\cdot\;,\o)}_{\var} &= \sup_{t\leq T} E^{Q^\o} \(M_t \exp \(\int^t_0 e^0 (Y_s) \d s\)\)\\
&\leq \exp \(T\norm{e^0}_\9 \) \sup_{t \le T} E^{Q^\o} \(M_t\).
\end{align*}
Taking the expectation with respect to $P$ it implies
\begin{eqnarray*}
E^P\(\sup_{t\leq T} \norm{\Gamma^Y_t (\cdot\;,\o)}_{\var} \)
&\le& \exp \( T\norm{e^0}_\9 \) E^P\( \sup_{t\leq T} E^{Q^\o} 
\(M_t\)\)\\
&\le &  \exp \(T\norm{e^0}_\9 \) E^P \(E^{Q^\o}\(\sup_{t\leq T} M_t\)\).
\end{eqnarray*}
By the Burkholder-Davis-Gundy (BDG) inequality this is bounded by
\begin{eqnarray*}
3 \exp \(T\norm{e^0}_\9\) E^{\bf Q} \( \left\langle M \right\rangle^{\frac12}_T\) & \le &
 3 \exp \(T\norm{e^0}_\9\)
 E^{\bf Q} \(  \left[ \int^T_0 \d s  \sum_{i=1}^N M^2_s
 e^i (Y_s)^2\right]^\halb \) \\
& \le& {\mathrm C}(e,N,T) \left\{E^{\bf Q}  \left( \int^T_0 \d s M_s^2 
 \right)\right\}^\frac{1}{2} ,
\end{eqnarray*}
where the last inequality is due to Jensen's inequality;
 ${\mathrm C}(e,N,T)$ is a constant
 depending on $N,T$ and $e^i,\; i=0\dots N,$.
By Fubini's Theorem and item iii),
 we have
\begin{equation*}
 E^{\bf Q}  \left( \int^T_0 \d s  M_s^2 \right)
\le T\exp(3T \sum^N_{i=1}\Vert e^i \Vert_\infty). 
\end{equation*}
\end{enumerate}
\end{proof}

The lemma below gives a characterization  
of the $\mu$-weighted laws of a process $Y$ living on
an enlarged probability space.
\begin{lemma} \label{LMargU}
Let $Y$ (resp. $\tilde Y$) be a process on a suitable 
enlarged probability space $(\Omega_0,\shg,  {\bf Q})$
(resp.  $(\tilde \Omega_0, \tilde \shg, \tilde {\bf Q})$).
Set $W = (W^1, \ldots, W^N)$.
Suppose that the law of $(Y,W)$ under ${\bf Q} $ and
the law of $(\tilde Y, W)$ under $\tilde{\bf Q}  $ are the same.
Then, the $\mu$-marginal weighted laws of $Y$ under ${\bf Q} $
coincide a.s. with the   $\mu$-marginal weighted
 laws of $\tilde Y$ under $\tilde {\bf Q} $. 
\end{lemma}
\begin{proof}  \
Let $0 \le t \le T$.
Using the assumption, we deduce    that
for any  bounded continuous function $f: \R \rightarrow \R$,
and every $F \in \shf_t$, we have
\begin{equation} \label{EMargU}
E^{{\bf Q}} \left (1_F f(Y_t) \she_t\left(\sum_{i=0}^N 
\int_0^\cdot  e^i(Y_s) dW^i_s \right)\right) = 
E^{\tilde {\bf Q} } \left (1_F f(\tilde Y_t) \she_t\left(\sum_{i=0}^N 
\int_0^\cdot  e^i(\tilde Y_s) dW^i_s\right)\right). 
\end{equation}
To show this, using
classical regularization properties of It\^o integral,
see e.g. Theorem 2 in \cite{RVSem},
and  uniform integrability arguments, 
 we first  observe that 
$$ \she_t\left(\sum_{i=0}^N 
\int_0^\cdot  e^i(Y_s) dW^i_s\right)$$
is the limit in $L^2(\Omega_0, {\bf Q})$ of 
$$ \she_t \left(\sum_{i=0}^N 
\int_0^\cdot  e^i(Y_s) \frac{W^i_{s+\varepsilon} - W^i_s}{\varepsilon} \d s
\right).$$
A similar approximation property arises replacing
$Y$ with $\tilde Y$ and ${\bf Q} $ with $\tilde {\bf Q} $.
Then \eqref{EMargU} easily follows.\\ 
To conclude, it will be enough to show the existence
of a countable  
 family $(f_j)_{j \in \N}$ of bounded continuous real functions
for which, for $P$ almost all $\omega \in \Omega$, 
for any $j \in \N$, we have $R_j = \tilde R_j$
where 
\begin{eqnarray*}
R_j(\omega) &=& E^{Q^\omega} \left(f_j(Y_t(\cdot,\omega)) 
 \she_t\left(\sum_{i=0}^N 
\int_0^\cdot  e^i(Y_s(\cdot,\omega)) dW^i_s\right)\right) \nonumber\\
&& \\
 \tilde R_j(\omega) &=& E^{\tilde Q^\omega} \left(f_j(\tilde Y_t(\cdot,\omega)) 
 \she_t\left(\sum_{i=0}^N 
\int_0^\cdot  e^i(\tilde Y_s(\cdot,\omega)) dW^i_s\right)\right) \nonumber.
\end{eqnarray*} 
This will follow, since applying \eqref{EMargU}, for any
$F \in \shf_t$, we have
$ E^P(1_F R_j) = E^P(1_F \tilde R_j)$.
\end{proof}

\subsection{SPDE, weak-strong existence of SDEs}
\label{WSE}

In this section we introduce the basic concepts related to
the stochastic porous media equation and the related
 non-linear diffusion.


 \begin{defi} \label{DSPDE}
 A  random field $X = (X(t, \xi, \omega), 
t \in [0,T], \xi 
 \in \R, \omega \in \Omega) $ is said to be a solution to \eqref{PME} if
 $P$ a.s. 
we have the following.
\begin{enumerate}
\item $X \in C([0,T]; \shs'(\R)) \cap  L^2([0,T]; L^1_{\rm loc} (\R))$.
\item $X$ is an $\shs'(\R)$
-valued $(\shf_t)$-progressively measurable process.
\item 
For any test function $\varphi \in \shs(\R)$ with compact support,
 $ t \in ]0,T]$ 
 a.s. we have 
\begin{eqnarray} \label{EDist}
\int_\R X(t,\xi) \varphi(\xi) \d\xi &=& \int_\R x_0(\d\xi) \varphi(\xi)  +
\halb \int_0^t \d s \int_\R 
\psi(X(s,\xi,\cdot))
 \varphi''(\xi) \d\xi
 \nonumber\\
& & \\
&+& \int_{[0,t] \times \R} X(s,\xi) \varphi(\xi) \mu(\d s,\xi) \d\xi.
\nonumber
\end{eqnarray}
\end{enumerate}
\end{defi}

At Definition \ref{DDoubleStoch}, we will present the concept of
{\it double stochastic non-linear diffusion} which is a McKean type 
equation with a supplementary source of randomness. 
Before this, as a first step, we will introduce a particular the case of simple
{\it double stochastic differential equation} (DSDE).
Let $\gamma:[0,T]\times \R \times \O \to \R$ be an $(\shf_t)$-progressively measurable random fields and
 $x_0$ be a  probability on $\shb(\R).$
\begin{defi}\label{DWeakStrong}
\begin{enumerate}
\item[a)] We say that (DSDE)$(\gamma, x_0)$ admits {\bf weak-strong existence}
 if there is a suitable
 extended probability space $(\Omega_0,\shg, {\bf Q} )$, i.e.
a  measurable space $(\Omega_1,\shh)$, a
 probability kernel $\(Q(\cdot\;,\o), \o\in\O \)$ on $\shh \times \O$, 
two ${\bf Q}$-a.s. continuous processes $Y,B$ on $(\O_0,\shg)$ 
where $\O_0=\O_1\times \O,\; \shg = 
\shh \otimes \shf$ such that the following holds.
\begin{description}
\item{1)} For almost all $\o$, $Y(\cdot,\o)$ is a (weak) solution to
\begin{equation} \label{DWS}
\begin{cases}
Y_t (\cdot\;,\o)=Y_0 + \int_0^t \gamma (s,Y_s (\cdot\;,\o),\o) \d B_s 
(\cdot\;,\o), \\
\text{Law} (Y_0)=x_0,
\end{cases} 
\end{equation}
with respect to $Q^{\o}$, where $B(\cdot\;,\o)$ is a $Q^\o$-Brownian motion
for almost all $\o$.

\item{2)} We denote $(\shy_t)$ the canonical  filtration
associated with $(Y_s, 0 \le s \le t)$ and
$\shg_t = \shy_t \vee (\{\emptyset, \Omega_1\} \otimes \shf_t)$.
 $W^1, \ldots, W^N$ is a $(\shg_t)$-martingale 
under ${\bf Q}  $.

\item{3)}
For every $0 \le s \le T$, for every bounded continuous 
$\sha: C([0,s]) \rightarrow \R$,
the r.v. $\o \mapsto E^{Q^\o}(\sha(Y_r(\cdot,\omega), r \in [0,s])) $
is $\shf_s$-measurable.

\end{description}
\item[b)] We say that (DSDE)$(\gamma,x_0)$ admits 
{\bf weak-strong uniqueness} if the following holds. 
Consider a measurable space $(\O_1,\shh)$ (resp. $(\wt \O_1, \wt \shh)$), 
a probability kernel $(Q(\cdot\;,\o),\o\in\O)$ (resp.  $(\wt 
Q(\cdot\;,\o),\o\in\O)$), with processes $(Y,B)$ (resp. $(\wt Y, \wt B)$) 
such that 
(\ref{DWS}) 
holds (resp.
(\ref{DWS}) 
  holds with $(\Omega_0,\shg,{\bf Q}  )$ replaced with $(\wt \Omega_0,
\wt \shg_0, \wt {\bf Q} ,$ $\wt {\bf Q}  $ being 
associated with $(\wt Q(\cdot\;, \o)$)).
Moreover we  suppose that item 2) is verified for $Y$ and $\tilde Y$.
\\
Then $(Y,W^1, \ldots, W^N)$ and $(\wt Y, W^1, \ldots, W^N)$ have the same law.
\item[c)] A process $Y$ fulfilling 
items 1) and 2) under (a)
will be called 
{\bf weak-strong solution of} (DSDE)$(\gamma,x_0)$.
\end{enumerate}
\end{defi}
\begin{rem} \label{R2.10}
Let $Y$ be a weak-strong solution of (DSDE)$(\gamma,x_0)$
with corresponding $B$.
\begin{description}
\item{a)} Since for almost all $\omega \in \Omega$, 
$B(\cdot,\omega)$ is a Brownian motion under $Q^\omega$,
it is clear that $B$ is a Brownian motion under $Q$,  
which is independent of $\shf_T$, i.e. independent
of $W^1, \ldots, W^N$. \\
Indeed let $\sha: C([0,T]) \rightarrow \R$ be a continuous
bounded functional, and denote by $\shw$ the Wiener measure
on  $C([0,T])^N$.
Let $F$ be a bounded  $\shf_T$-measurable r.v.
 Since for each $\omega$, $B(\cdot,\omega)$ is
a Wiener process with respect to $Q^\omega$,
we get 
\begin{eqnarray*}
 E^{\bf Q}  (F \sha(B)) &=& \int_\Omega F E^{Q^\omega}(\sha(B(\cdot,\omega))) \d P(\omega)
= \int_\Omega F(\omega) \d P(\omega) \int_{\Omega_1} \sha(\omega_1)
 d\shw(\omega_1) \\
 &=& \int_{\Omega_0} F(\omega) \d {\bf Q}(\omega_0)
 \int_{\Omega_0} \sha(\omega_1) \d {\bf Q}(\omega_0).
\end{eqnarray*}
This shows that $(W^1,\ldots,W^N)$ and $B$ are independent.
Taking $F = 1_\Omega$ in previous expression, the equality between the
left-hand side and the third term, shows that 
$B$ is a Brownian motion under $Q$.
\item{b)} Since for any $ 1\le i,j \le N$,
\begin{equation}\label{EIJ}
 [W^i,W^j]_t = \delta_{ij} t,  \ [W^i, B]= 0, \ [B,B]_t = t,
\end{equation}
 L\'evy's characterization theorem, implies that   
$(W^1,\dots,W^N, B)$ is a ${\bf Q}  $-Brownian motion.
\item{c)} An equivalent formulation to 1) 
in item a) of Definition \ref{DWeakStrong} is the following. 
For $P$ a.e.,  $\o \in  \Omega$, $Y(\cdot\;,\o)$ solves the $Q^\o$-martingale problem with respect to the (random) PDE operator
\begin{equation*}
L^\o_t f(\xi) = \frac{1}{2} \gamma^2 (t,\xi,\o) f''(\xi),
\end{equation*}
and initial distribution $x_0$. Indeed, we remark that the construction
can be performed on the canonical space $\Omega_1 = C([0,T];\R)$.
\end{description}
\end{rem}

\begin{prop} \label{PB1}
Let  $Y$ be a process  as in Definition \ref{DWeakStrong} a). 
We have the following.
\begin{enumerate}
\item $Y$ is a $(\shg_t)$-martingale
on the product space $(\Omega_0,\shg, {\bf Q})$.
\item $[Y, W^i] = 0, \ \forall 1 \le i \le N$. 
\end{enumerate}
\end{prop}
\begin{proof} Let $0 \le s < t \le T$, $F_s \in \shf_s$ and
 $G: C([0,s]) \rightarrow \R$ be continuous and bounded. 
We will  prove below  that, for $1 \le i \le N+1$,
setting $W^{N+1}_t = 1$, for all $t \ge 0$,
\begin{equation} \label{EB1}
E^{\bf Q}  (Y_t W^i_t  G(Y_r, r \le s) 1_{F_s} ) = E^{\bf Q} (Y_s W^i_s 1_{F_s} 
 G(Y_r, r \le s)).
\end{equation}
Then \eqref{EB1} with $i= N+1$ shows  item 1.
Considering \eqref{EB1} with $1 \le i \le N$,
shows that $Y W^i$ is a $(\shg_t)$-martingale, which shows
item 2.
Therefore, it remains to show \eqref{EB1}.\\
The left-hand side of that equality gives
\begin{eqnarray*}
\int_\Omega \d P(\omega) &&  W^i_t(\omega) 1_{F_s}(\omega)  E^{Q^\omega} 
\left(Y_t(\cdot,\omega)  G(Y_r(\cdot,\omega), r \le s) \right) \\
&=&
\int_\Omega \d P(\omega) 1_{F_s}(\omega) W^i_t(\omega) 
 E^{Q^\omega} \left(Y_s(\cdot,\omega) G(Y_r(\cdot,\omega), r \le s)\right),
\end{eqnarray*}
because $Y(\cdot, \omega)$ is a $Q^\omega$-martingale for $P$-almost all
 $\omega$.
To obtain the right-hand side of   \eqref{EB1}
 it is enough to remember that  $W^i$ are 
$(\shg_t)$-martingales and that item a) 3) in Definition \ref{DWeakStrong}
holds.
This concludes the proof of Proposition \ref{PB1}.
\end{proof}
\begin{rem} \label{RMargU}
Lemma \ref{LMargU} shows that, whenever weak-strong uniqueness holds, then the
 $\mu$-weighted marginal laws  of any weak solution $Y$ are uniquely determined.
\end{rem}

\section{The concept of doubly probabilistic representation}
\label{S3}



\subsection{The doubly stochastic non-linear diffusion.}
\label{S32}

\setcounter{equation}{0}

We come back to the notations and conventions of the introduction and of Section \ref{S2}.
Let $x_0$ be a probability on $\R$.
The  doubly probabilistic representation is based on the following idea.
Let  $Y: \Omega_0 \times [0,T] \rightarrow \R$ be a measurable process
where $\Omega_0 = \Omega_1 \times \Omega$
is the usual enlarged probability space as introduced in
Definition \ref{DSEPS}. Let ${\bf Q}  $ be a probability inherited from
a random kernel $Q^\omega$ as before Definition \ref{DSEPS}.
Let $ (\shg_t),$ where $(\shg_t)$ is some filtration on $(\Omega_0,\shg)$
such that $W^1, \ldots, W^N$ are $(\shg_t)$-Brownian motions on 
$(\Omega_0,\shg,{\bf Q})$.

Suppose that 
\begin{equation}
\label{DPIY}
\left \{
\begin{array}{ccc}
Y_t &=& Y_0 + \int_0^t \Phi(X(s,Y_s)) dB_s, \\
\mu-{\rm Weighted \ Law } (Y_t) &=& X(t,\xi) \d  \xi, \quad t \in ]0,T],\\
\mu-{\rm Weighted \ Law } (Y_0) &=& x_0(\d \xi), 
\end{array}
\right.
\end{equation}
where $B$ is a $Q$-standard Brownian motion with respect to $ (\shg_t).$
Then  $X$  solves the SPDE \eqref{PME}. This will be the object of Theorem \ref{T32}.
Vice versa, if $X$ is a solution of \eqref{PME} then there is a process $Y$
solving \eqref{DPIY}, see Theorem \ref{T73}.


\setcounter{equation}{0}

\begin{defi}\label{DDoubleStoch}
\begin{enumerate}
\item[1)] We say that the doubly stochastic non-linear diffusion (DSNLD)
 driven by $\Phi$ (on the space $(\O,\shf,P)$ with initial condition $x_0$,
 related to the random field $\mu$ (shortly (DSNLD)$(\Phi, \mu, x_0)$)
 admits {\bf weak existence} if there is 
a measurable  random field $X : [0,T]
 \times \R \times \O \to \R$ with the following properties.
\subitem a) 
The problem (DSDE)$(\gamma,x_0)$  with
 $\gamma(t,\xi,\omega) = \Phi(X(t,\xi,\omega))$
admits weak-strong existence.
\subitem b) $X = X(t,\xi, \cdot) \d \xi, t \in ]0,T]$,
 is the family of $\mu$-marginal weighted laws of $Y$, where
$Y$ is the solution of \eqref{DWS} in Definition \ref{DWeakStrong}.
 In other words
$X$ constitutes the densities of those $\mu$-marginal weighted laws.
\item[2)] A couple $(Y,X)$, such that $Y$ is a (weak-strong) solution to the \\
(DSDE)$(\gamma,x_0)$,
 is  called {\bf weak solution} to the (DSNLD)$(\Phi,\mu,x_0)$. 
 $Y$ is also called doubly stochastic representation of the random field $X$.
\item[3)]
Suppose that, given two measurable 
 random fields $X^i: [0,T] \times \R \times \O\to \R,
 i = 1,2$ on $(\O, \shf, P, (\shf_t))$,
  and $Y^i$, on extended probability space $(\Omega_0^i, {\bf Q}^i), i=1,2$,
such that $(Y^i, X^i)$ is a
  weak-strong solution  of 
(DSDE)$(\Phi (X^i),x_0), i =1,2$, where
we  always have that $(Y^1, W^1, \ldots, W^N)$ 
and   $(Y^2, W^1, \ldots, W^N)$ have the same law. Then 
we say that the (DSNLD)$(\Phi,\mu,x_0)$ admits
 {\bf weak uniqueness}.
\end{enumerate}
\end{defi}
\begin{rem} \label{RRRR}
If (DSNLD)$(\Phi,\mu,x_0)$ admits {\bf weak uniqueness} then 
the $\mu$-marginal weighted laws of $Y$
 are uniquely determined, $P$-a.s., see Lemma \ref{LMargU}.
\end{rem}
\begin{theo} \label{T32}
Let $(Y,X)$ be a solution of (DSNLD)$(\Phi,\mu,x_0)$. 
Then $X$
 is a solution to the SPDE \eqref{PME}.
\end{theo}

\begin{rem} \label{RDPIY}
\begin{enumerate}
\item Let $t \in [0,T]$. Let $\varphi: \R \rightarrow \R$ be
Borel and bounded. Then
$$ \int_\R \varphi(\xi)  X(t,\xi,\omega) \d \xi = 
 E^{Q^\omega}\left(\varphi(Y_t(\omega)) \she_t\left(\int_0^\cdot \mu(ds, Y_s(\omega))\right)
\right).$$
So 
$$ \int_\R   X(t,\xi,\omega) d\xi = 
 E^{Q^\omega} \left(\she_t \left(\int_0^\cdot \mu(ds, Y_s(\omega))\right)
\right).$$
Even though  for a.e. $\omega \in \Omega$,  the   previous expression
is not necessarily a probability measure, of course, 
$$ \nu_\omega: \varphi \mapsto \frac{\int_\R \varphi(\xi) 
 X(t,\xi,\omega) d\xi}
{\int_\R  X(t,\xi,\omega) d\xi } $$
is one. It can be expressed as 
$$\nu_\omega(A) =  \frac{E^{Q^\omega}(1_A(Y_t) \she_t(M(\cdot, \omega)))}
{E^{Q^\omega} \she_t(M(\cdot,\omega))}, $$
where
$M_t(\cdot,\omega) = \int_0^t \mu(ds, Y_s(\cdot, \omega)), t \in [0,T],$
 is defined in \eqref{1.2bis}.
\item Consider the particular case $e_0 = 0, e_1 = c$, 
$c$ being some constant. In this case, the $\mu$-marginal laws are given by
\begin{equation*}
A \mapsto  E^{Q^\omega}(1_A(Y_t) c \she_t(W)) = 
 c \she_t(W) E^{Q^\omega}(1_A(Y_t)) =
 c \she_t(W) \nu_\omega(t,A)
\end{equation*}
and $\nu_\omega(t,\cdot)$ is the law of $Y_t(\cdot,\omega)$ under 
$Q^\omega$.
\end{enumerate}
\end{rem}

\begin{proof}

Let $B$ denote the Brownian motion associated to $Y$ as 
a solution   to  (DSDE)$(\gamma,x_0)$,
mentioned in item a)1)  of Definition \ref{DDoubleStoch}.
For $t \in [0,T]$, we set
$$
Z_t = \she_t\(\int^\cdot_0 \mu (\d s,Y_s)\), \quad
M_t = Z_t \exp \left(-\int_0^t e^0(Y_s) \d s \right), \ t\in [0,T]. 
$$

\begin{enumerate} 
\item Proof of  Definition \ref{DSPDE} 1. \
By Proposition \ref{P28}, $(M_t, t \in [0,T])$ is a uniformly
integrable martingale. Consequently $t \mapsto Z_t$ is continuous
in $L^1(\Omega_0, {\bf Q})$. On the other hand the process $Y$ is continuous. 
This implies that $P$ a.e. $\omega \in \Omega$, $X \in C([0,T]; \shm(\R))$,
where $\shm(\R)$ is equipped with the weak topology.
This  implies that   $X \in C([0,T]; \shs'(\R))$.
Furthermore, for $P$ a.e. $\omega \in \Omega$, and $t\in ]0,T]$,
$X(t,\cdot, \omega) \in L^1(\R)$ and 
$\int_\R X(t,\xi,\omega) \d \xi = \Vert \Gamma(t,\cdot,\omega)\Vert_{\rm var}$.
By  Proposition \ref{P28} iv), it follows that $P$-a.s.
$ X(\cdot,\cdot,\omega)  \in L^\infty([0,T];L^1(\R)) \subset L^2([0,T];L^1_{\rm loc}(\R)).$ 
\item  Definition \ref{DSPDE} 2. follows from Remark \ref{RDPIY} 1) and Definition \ref{DWeakStrong} a) 3).
\item Proof of  Definition \ref{DSPDE} 3. \
Let $\varphi\in \shs(\R)$ with compact support.
Taking into account Proposition \ref{PB1}, we apply It\^o's formula to get
\begin{align*}
&\varphi(Y_t) Z_t = \varphi(Y_0)+\int^t_0\varphi' (Y_s) Z_s \d Y_s
+\int^t_0 \varphi(Y_s) Z_s\(\mu (\d s,Y_s) -\frac12 \sum^N_{i=1} (e^i(Y_s))^2 \d s \)\\
&+\frac12 \int^t_0 \varphi''(Y_s) \Phi^2 (X(s,Y_s)) Z_s \d s
+\frac12 \int^t_0 \varphi(Y_s) Z_s \left( \sum_{i=1}^N (e^i (Y_s))^2 \right) \d s.
\end{align*}
Indeed we remark that
$ \int^t_0 \varphi' (Y_s) \d [Z,Y]_s = 0,$
because \\
$[Z,Y]_t = \sum_{i=1}^N \int^t_0 e^i (Y_s) Z_s  \d [W^i,Y]_s = 0;$
in fact $[W^i,Y]=0$ by Proposition \ref{PB1}.
So
\begin{eqnarray*}
 \varphi (Y_t) Z_t &=& \varphi(Y_0) + \int^t_0 \varphi' (Y_s) Z_s \Phi (X(s,Y_s)) \d B_s\\
 &+& \int^t_0 \varphi (Y_s) Z_s \mu (\d s, Y_s)
+\frac12 \int^t_0 \varphi'' (Y_s) \Phi^2 (X(s,Y_s)) Z_s \d s.
\end{eqnarray*}
Taking the expectation with respect to $Q^\o$ we get $\pas$,
\begin{align*}
\int_\R \d \xi \varphi (\xi) X(t,\xi) &= \int_\R \varphi(\xi)x_0(\d\xi) +
\sum_{i=0}^N \int^t_0 \d W_s^i
 \( \int_\R \d \xi \varphi(\xi) e^i (\xi) X(s,\xi)\)\\
&+ \frac12 \int_0^t \d s \int_\R \d \xi \varphi'' (\xi) \Phi^2 (X(s,\xi)) X(s,\xi),
\end{align*}
which implies the result. Indeed, in the previous equality,  we have used  
Lemma \ref{Lmulaw} below.
\end{enumerate}
\end{proof}

\begin{lemma} \label{Lmulaw}
Let $ 1 \le i \le N$.
For $P$ a.e. $\omega \in \Omega$, we have
$$ E^{Q^\omega}\left( \int^t_0 \varphi (Y_s) Z_s e^i(Y_s) \d W^i_s \right)
(\cdot,\omega)  
= \int_0^t \d W^i_s(\omega) \int_\R \varphi(\xi) e^i(\xi) X(s,\xi,\omega) \d \xi.
$$
\end{lemma}
\begin{proof}
Since the Brownian motions $W^i$ are not random for $Q^\omega$, it is 
possible to justify the 
permutation of the stochastic integral with respect to $W^i$
and $E^{Q^\omega}$ by a Fubini argument  approximating the
stochastic integrals via Lebesgue integral, see e.g. Theorem 2 of
\cite{RVSem}. A complete proof is given in \cite{Bar-Roc-Rus-2200}. 
\end{proof}

\subsection{Filtering interpretation}

\label{SFiltering}

Item 1. of Remark \ref{RDPIY} has an interpretation in the framework of 
filtering theory, see e.g. \cite{pardoux} for a comprehensive introduction
on that subject. \\
Suppose $e^0 = 0$.
Let $\hat{\bf Q}$ be a probability on some probability space $(\Omega_0, \shg_T)$,
and consider the non-linear diffusion problem \eqref{E1.2}
as a basic dynamical phenomenon. We suppose now  
that there are $N$ observations $Y^1,\ldots,Y^N$ related to  the process $Y$
generating a filtration $(\shf_t)$.
We suppose in particular 
that $\d Y^i_t = \d W^i_t + e^i(Y_t) \d t, 1 \le i \le N,$
and  $W^1, \ldots, W^N$ be $(\shf_t)$-Brownian motions.
Consider the following dynamical system of non-linear diffusion type:
\begin{equation}\label{EFiltering1}
\left \{
 \begin{array}{ccc}
Y_t &=& Y_0 + \int_0^t \Phi(X(s,Y_s))dB_s \\
\d Y^i_t &=& \d W^i_t + e^i(Y_t) \d t, 1 \le i \le N,\\
X(t,\cdot)&:& {\rm conditional \ law \ of} \ Y_t \  {\rm under} \ \shf_t.  
\end{array}
\right.
\end{equation}
The third equality of \eqref{EFiltering1} means, under $\hat  {\bf Q} $,
 that we have,
\begin{equation} \label{EFiltering2}
\int_\R \varphi(\xi) X(t,\xi) \d \xi = 
E(\varphi(Y_t) \vert \shf_t).
\end{equation}
We remark that, under the new probability $Q$
defined by $\d {\bf Q} = \d \hat {\bf Q} \she(\int_0^T \mu(\d s, Y_s))$,
 $Y^1, \ldots, Y^N$ are standard $(\shf_t)$-independent Brownian motions.
Then \eqref{EFiltering2} becomes
$$
\int_\R \varphi(\xi) X(t,\xi) \d \xi = E^{\hat {\bf Q}}(\varphi(Y_t) \vert \shf_t) 
=  \frac{E^{{\bf Q}}(\varphi(Y_t)   \she_t (\int_0^\cdot \mu(ds, Y_s) \vert \shf_t ))}
{E^{{\bf Q}}(\she_t (\int_0^\cdot \mu(ds, Y_s) \vert \shf_t))}.
$$
Consequently, by Theorem \ref{T32}, $X$ will be the solution 
of the SPDE \eqref{PME}, with $x_0$ being the law of $Y_0$; so  
\eqref{PME} constitutes
the Zakai type equation associated with our filtering problem.

\section{The densities of the $\mu$-marginal weighted laws}

\label{S4}

\setcounter{equation}{0}

This section constitutes an important step towards the doubly probabilistic representation
of a solution to \eqref{PME}, when $\psi$ is non-degenerate. 
Let $x_0$ be a fixed probability on $\R$.
We recall that a process $Y$ (on a suitable enlarged probability space $(\O_0, \shg, {\bf Q})$), which 
is a weak solution to the (DSNLD)$(\Phi,\mu,x_0)$, is in particular a weak-strong solution of a (DSDE)$(\gamma,x_0)$ 
where $\gamma:[0,T] \times \R \times \O \rightarrow \R$ is some suitable 
progressively measurable random field on $(\O,\shf,P)$.
The aim of this section is twofold.
\begin{enumerate}
\item[A)] To show that whenever $\gamma$ is a.s. bounded and non-degenerate, (DSDE)$(\gamma,x_0)$ admit weak-strong existence and uniqueness.
\item[B)] The marginal $\mu$-laws of the solution to (DSDE)$(\gamma,x_0)$ 
admit a density for $P \; \o$ a.s.
\item[A)] We start discussing well-posedness.
\end{enumerate}
\begin{prop} \label{P4.1}
We suppose the existence of random variables $A_1, A_2$ such that
\begin{equation} \label{he4.2}
0< A_1 (\o) \leq \gamma (t,\xi,\o) \leq A_2 (\o), \quad 
\forall (t,\xi) \in [0,T] \times \R,     \quad \pas
\end{equation}
Then (DSDE)$(\gamma,x_0)$ admits weak-strong existence and uniqueness.
\end{prop}
\begin{proof}
\textit{Uniqueness.} This is the easy part. Let $Y$ and $\wt Y$ be two solutions. Then for $\o$ outside a $P$-null set $N_0, Y(\cdot\;, \o)$ and $\wt Y (\cdot\;, \o)$ are solutions to the same one-dimensional classical SDE with measurable bounded and non-degenerate (i.e. greater than a strictly positive constant) 
coefficients. Then, by Exercise 7.3.3
of  \cite{sv} the law of $Y(\cdot\;,\o)$ equals the law of $\wt Y(\cdot\;,\o)$. Then obviously the law of $Y$ equals the law of $\wt Y$.

\medskip
\textit{Existence.} This point is more delicate. In fact one needs 
to solve the random SDE for $P$ almost all $\o$  but in such
a way that the solution produces bimeasurable processes $Y$ and $B$.

First we regularize the coefficient $\gamma$. Let $\phi$ be a mollifier
 with compact support; we set
$\phi_n (x) = n\phi (nx), \; x\in \R \; , \; n\in \N.$
We consider the random fields $\gamma_n : [0,T]
 \times \R \times \O \to \R$ by 
 $\gamma_n (t,x,\o) := \int_\R 
\gamma (t,x-y,\o) \phi_n (y) \d y$. 
\\
Let $(\wt \O_1, \wt \shh_1,\wt P)$ be a probability space where
 we can construct a random variable $Y_0$ distributed according to $x_0$ 
and an independent Brownian motion $B$.

In this way on $({\wt \O}_1 \times \O, \wt \shh_1 \otimes \shf, \wt P \otimes P)$ we dispose of a random variable $Y_0$ and a Brownian motion independent of
 $\{\emptyset,{\wt \O}_1\} \otimes \shf$. By usual fixed point techniques, it is possible to exhibit a (strong) 
solution of (DSDE)$(\gamma_n,x_0)$ on the 
over mentioned product probability space.
We can show that there is a unique solution  $Y = Y^n$ of
$Y_t  = Y_0 + \int^t_0 \gamma_n (s,Y_s,\cdot) \d B_s. $
In fact, the maps 
$\Gamma_n : Z \mapsto \int_0^\cdot \gamma_n (s,Z_s,\o) \d B_s + Y_0,$
where $\Gamma_n : L^2(\wt \O_1 \times \O; \; \wt P \otimes P) \to L^2 (\wt \O_1 \times \O, \wt P \otimes P)$ are Lipschitz; by usual Picard fixed point arguments one can show the existence of a unique solution $Z=Z^n$ in $L^2(\wt \O_1 \times \O; \; \wt P \otimes P)$.
We observe that, by usual regularization arguments for It\^o integral
 as in Lemma \ref{Lmulaw},
for $\omega$-a.s.,
$Y(\cdot,\omega)$ solves for $P$ a.e. $\omega \in \Omega$, equation
\begin{equation}\label{e4.1}
Y_t (\cdot\;,\o) = Y_0 + \int^t_0 \gamma_n (s,Y_s (\cdot\;,\o),\o) \d B_s,
\end{equation}
on $(\wt \O_1, \wt \shh_1,\wt P)$.
We consider now the measurable space $\O_0 = \O_1 \times \O$, where 
$\O_1 = C([0,T], \R),$ 
equipped with product $\sigma$-field $\shg = \shb(\Omega_1) \otimes \shf$.
On that measurable space, we introduce  the
 probability measures ${\bf Q}_n$ where 
${\bf Q}_n(\d \o_1, \o) =  Q_n (\d \o_1, \o) P(\d \o)$
and
 $Q_n (\cdot, \o)$ being the law of $Y^n(\cdot\;,\o)$ for almost all fixed $\o$.
\\
We set $ Y_t(\o_1,\o)=\o_1(t)  $,
where $\omega_1 \in C([0,T];\R)$.
We denote by $( \shy_t, \; t\in[0,T])$ (resp. $(\shy^1_t)$)
 the canonical 
 filtration associated with $Y$ on $\Omega_0$ (resp. $\Omega_1$).
 The next step will be the following.

 \begin{lemma} \label{Lconv}
For almost all $\o  \ dP$ a.s.
$Q_n (\o,\;\cdot)$ converges weakly to $Q (\o, \cdot)$, where under
 $Q(\cdot, \o), \; Y(\cdot\;, \o)$ solves the SDE
\begin{equation}\label{e4.2}
Y_t(\cdot\;, \o) = Y_{0} + \int^t_0 \gamma (s,Y_s(\cdot\;, \o),\o) 
\d B_s(\cdot,\omega),
\end{equation}
where $B(\cdot, \omega)$ is an $(\shy^1_t)$-Brownian motion 
on $\Omega_1$.
\end{lemma}
\begin{proof} \
It follows directly from Proposition \ref{A4} of the Appendix.
\end{proof}
This shows the validity of 1) if Definition \ref{DWeakStrong} a).

\begin{rem} \label{RConv1}
\begin{enumerate}
\item[1)]
 Since $Q_n (\cdot, \o)$ converges weakly to $Q (\cdot, \o)$,
 $\o \ dP$ a.s., then the limit (up to an obvious modification) is a measurable random kernel.
\item[2)] This also implies that $Y_n(\cdot, \o)$ converges stably to
 $Q (\cdot, \o)$. 
For details about the stable convergence the reader can consult 
\cite[section VIII 5. c]{jacod} and the recent monograph \cite{Hausler}.
\end{enumerate}
\end{rem}
The considerations above allow to complete the proof of Proposition \ref{P4.1}.
By Lemma \ref{Lconv}, $Q^\omega  = Q(\cdot,\omega)$ is
a random kernel, being a limit of random kernels.
Let us consider  the associated probability measure on the suitable
enlarged probability space $(\Omega_0,\shg,Q)$.
 We observe that $Y$ on $(\O_0,\shg)$
 is obviously measurable, because it is the canonical process $Y(\o_1,\o)=\o_1$.
Setting
\begin{equation*}
B_t(\cdot, \o)  = \int_0^t \frac{\d Y_s(\cdot,\omega)}{\gamma (s, Y_s(\cdot,\omega),\o)} , 
\end{equation*}
we get $[B]_t(\cdot, \o) =t$ under $Q(\cdot\;, \o)$, so,
by L\'evy characterization theorem,  it is a Brownian motion.
 Moreover $B$ is bimeasurable. \\
Let $G= \sha(Y_r(\cdot,\omega), r \in [0,s])$, where $\sha$ is a bounded
functional $C([0,s]) \rightarrow \R$.
We first observe that the r.v. 
$\omega \mapsto E^{Q^\omega}(G)$
is $\shf_s$-measurable. 
This happens because $Y$ is, under $Q^\omega$, 
a martingale with quadratic variation \\
$\left(\int_0^t \gamma^2(s, Y_s(\cdot, \omega),\omega) \d s, 0 \le t \le 
T\right)$, 
i.e. with (random) coefficient which is $(\shf_t)$-progressively measurable. 
This shows item 3) of Definition \ref{DWeakStrong} a).

The last point to check is that 
$W^1,\ldots, W^N$ are $(\shg_t)$-martingales,
where $\shg_t =  \shy_t \vee (\{\emptyset, \Omega_1\} \otimes \shf_t), \ 0 \le t \le T$, i.e. item 2) of Definition \ref{DWeakStrong}.
 \\
Indeed, we justify this immediately. Consider $0 \le s \le t \le T$.
Taking into account monotone class arguments, given $F \in \shf_s$,
$G \in \shy^1_s$, $1 \le i \le N $, 
it is enough to prove that
\begin{equation} \label{ECondExp}
 E^{\bf Q} ( F G W^i_t) = E^{\bf Q}(F G W^i_s).
\end{equation}
Using the fact that
 $W^i$ is an $(\shf_t)$-martingale and that $E^{Q^{\omega}} (G)$
is $\shf_s$-measurable by item a) 3) of Definition \ref{DWeakStrong}
(established above),
 the left-hand side of \eqref{ECondExp} gives
$$
 E^P ( F W^i_t  E^{Q^{\omega}} (G)) =
  E^P ( F W^i_s  E^{Q^{\omega}} (G)), 
$$
which constitutes the right-hand side of
\eqref{ECondExp}. 
This concludes the proof of the proposition.
\end{proof}

We go on now with step B) of the beginning of Section \ref{S4}.
\begin{prop} \label{P4.2}
We suppose the existence of r.v. $A_1,A_2$ such that
\begin{equation} \label{4.2bis}
0 < A_1(\o)\leq \gamma (t,\xi,\o)\leq A_2(\o), \forall
(t,\xi) \in [0,T] \times \R, \quad \text{ a.s.}
\end{equation}
Let $Y$ be a weak-strong solution to (DSDE)$(\gamma,x_0)$
and we denote by $(\nu_t(dy,\cdot), t \in [0,T])$, the $\mu$-marginal weighted
 laws of process $Y$. 
\begin{enumerate}
\item  There is a  measurable function $q: [0,T] \times \R \times \Omega \rightarrow \R_+$
such that  $\d t \d P $ a.e., 
$\nu_t(\d y,\cdot) = q_t(y,\cdot) \d y$.
In other words the  $\mu$-marginal weighted laws admit densities.
\item
$\int_{[0,T]  \times \R} q^2_t (y,\; \cdot) \d t \d y < \9 \quad \pas.$
\item $q$ is an $L^2(\R)$-valued progressively measurable process.  
\end{enumerate}
\end{prop}
\begin{proof}
By 3) of Definition \ref{DWeakStrong}, the $\mu$-marginal laws constitute
an $\shs'(\R)$-valued progressively measurable process.
Consequently  3. holds if 1. and 2. hold.

Let
\begin{equation*}
B_t(\cdot,\omega) := \int_0^t \frac{\d Y_s(\cdot,\omega) }{\gamma(s,Y_s(\cdot,\omega),\o)}.
\end{equation*}
We denote again $Q^\o:=Q(\cdot\;, \o)$ according to Definition \ref{DWeakStrong}, $\omega \in \Omega$.\\
Let $\o \in\O$ be fixed. Let $\varphi:[0,T] \times \R\to \R$ be a continuous function with 
compact support.
We need to evaluate
\begin{equation}\label{e4.4}
E^{Q^\o} \( \int_0^T \varphi (s,Y_s) Z_s \d s \),
\end{equation}
where $Z_s = M_s \exp \(\int_0^s e^0 (Y_r)\d r \) $ where
 $M_s = \she_s \(\sum^N_{i=1} \int^\cdot_0 e^i (Y_r)\d W_r^i \).$ \\
$M_s$ is smaller or equal than
$
 \exp \( \sum^N_{j=1} \int^s_0 e^j \(Y_r\) \d W^j_r \) 
$
which equals 
\begin{equation} \label{e4.4TER}
 \exp \( \sum^N_{j=1} \left\{ W_s^j e^j(Y_s) -
 \int^s_0 W^j_r (e^j)'(Y_r) 
\d Y_r \right \} 
  -\frac{1}{2} \int_0^s \left \{ \sum^N_{j=1} W_r^j (e^j)''(Y_r)  
\gamma^2(r, Y_r, \cdot)   
\right  \} \d r \), 
\end{equation}
taking into account the fact that $[Y, W^j] = 0$ for any $1 \le j \le n$,
by Proposition \ref{PB1}. 
Denoting $\Vert g \Vert_\infty : = \sup_{t \in [0,T]} \vert g(t) \vert,$
for a function $g:[0,T] \rightarrow \R$,
\eqref{e4.4TER} is  smaller or equal than
$$ \exp \( \sum^N_{j=1}  \Vert W^j \Vert_\infty
 (\norm{e^j}_\infty + \frac{T}{2}  \norm{(e^j)''}_\infty A_2^2) \)
 \exp \( - \int_0^s \left[ \sum^N_{j=1}  W^j_r (e^j)'
 (Y_r)
\gamma (r,Y_r,\cdot)\right]\d B_r\).$$
So \eqref{e4.4} is bounded by
\begin{equation}\label{e4.5}
\varrho(\o) E^{Q^\o} \(\int^T_0 \vert \varphi \vert (s,Y_s(\cdot,\omega)
 R_s(\cdot,\omega) \d s\),
\end{equation}
where
\begin{eqnarray*}
\varrho(\o) &=& \exp \left(T\norm{e_0}_\9 + \sum^N_{i=1} \norm{W^i}_\9 
\norm{e^i}_\9 \right.\\
 &+&  \left. T \frac{A_2^2(\omega)}{2} \sum^N_{i=1} 
(\norm{W^i}^2_\9 \norm{(e^i)'}^2_\9 + \Vert W^i \Vert_\infty \Vert 
(e^i)''\Vert_\infty) \right)
\end{eqnarray*}
and $R$ is the $Q^\o$-exponential martingale
$$
R_t (\; \cdot \;, \o) = \exp \bigl( -\int^t_0 \delta (r, \; \cdot \;,\o) 
\d B_r  
  - \frac12 \int^t_0 \delta^2 (r,\;\cdot\;,\o)\d r \bigr).
$$
where
$ \delta (r,\;\cdot\;,\o) = \sum^N_{j=1} W_r^j (e^j)' \(Y_r(\; \cdot \; , \o)\)
 \gamma \(r,Y_r(\;\cdot\;,\o), \o \).$
So there is a random (depending on $\omega \in \Omega$) constant
\begin{equation}\label{e4.6}
\varrho_1 (\o) := {\rm const} \(T,W^j,\norm{e^j}_\9,\norm{(e^{j})'}_\9,
\norm{(e^{j})''}_\9,
 \; 1\leq j \leq N,\; A_2(\o)\),
\end{equation}
so that \eqref{e4.5} is smaller than
\begin{equation}\label{e4.7}
\varrho_1(\o) E^{Q^\o} \( \int^T_0 
\vert \varphi(s,Y_s(\;\cdot\;,\o))\vert \d s R_T (\;\cdot\;,\o) \),
\end{equation}
where we remind that $R(\cdot,\omega)$ is a $Q^\omega$-martingale.
By Girsanov theorem, \\
$\wt B_t(\cdot,\o) = B_t(\cdot, \o) + \int_0^t \delta(r,\;\cdot\;,\o)\d r$
is a $\wt Q^\o$-Brownian motion with
$ \d \wt Q^\o = R_T (\;\cdot\;,\o) \d Q^\o.$
At this point, the expectation in \eqref{e4.7} gives 
\begin{equation}\label{e4.8}
E^{\wt Q^\o} \(\int^T_0  \vert \varphi \vert(s,Y_s (\;\cdot\;,\o)) \d s \),
\end{equation}
where
$$
Y_t (\;\cdot\;,\o) = Y_0 + \int^t_0 \gamma (s,Y_s(\;\cdot\;,\o),\o)
 \d \wt B_s
- \int_0^t \gamma (s,Y_s (\;\cdot\;,\o), \o) \delta (s, \;\cdot\;,\o) \d s.
$$

For fixed $\o\in\O$, $\delta$ is bounded by a random constant $\varrho_2(\o)$ of the type \eqref{e4.6}.
Moreover we keep in mind assumption \eqref{e4.1} on $\gamma$.
By Exercise 7.3.3 of \cite{sv}, \eqref{e4.8} is bounded by
$\varrho_3(\o)\norm{\varphi}_{L^2 ([0,T]\times \R)},$
where $\varrho_3(\o)$ again depends on the same quantities as in \eqref{e4.6} and $\Phi$.
So for $\o \pas$,
 the map $\varphi \mapsto E^{Q^\o} \(\int_0^T \varphi
 (s,Y_s (\;\cdot\;,\o)) Z_s (\;\cdot\;,\o) \d s\right)$ prolongs to $L^2([0,T]\times \R)$. Using Riesz' theorem
 it is not difficult to show the existence of an $L^2([0,T]\times\R)$
function $(s,y)\mapsto q_s (y,\o)$ which constitutes indeed the density of the family of the $\mu$-marginal weighted laws.
\end{proof}

\section{On the uniqueness of a Fokker-Planck type SPDE}
\label{S5}

\setcounter{equation}{0}


The next result is an extension of Theorem 3.8 of \cite{BRR1} 
to the stochastic case.  
It has an independent interest  since it is
 a Fokker-Planck SPDE with possibly degenerate measurable coefficients.

\begin{theo} \label{T51}
Let $z_0$ be a distribution in $\shs'(\R)$.
Let $z^1, z^2$ be two measurable random fields belonging $\o$ a.s. to 
$C([0,T],\shs'(\R))$ such that $z^1, z^2: ]0,T]\times \O \to \shm (\R)$.
 Let $a:[0,T]\times\R\times\O\to\R_+$ be a bounded measurable random field
such that, for any $t \in [0,T]$, $a(t, \cdot)$ is
$\shb([0,t]) \otimes \shb(\R) \otimes  \shf_t$-measurable. 
We suppose moreover the following.
\begin{enumerate}
\item[i)] $z^1-z^2 \in L^2 ([0,T]\times \R)$ a.s.
\item[ii)]  $t \mapsto (z^1-z^2)(t,\cdot)$  is an
 $(\shf_t)$-progressively measurable
$\shs'(\R)$-valued process.
\item[iii)]   $ \int_0^T \Vert z^i(s,\cdot \Vert_{\rm var}^2 \d s < \infty$ a.s.
\item[iv)] $z^1,z^2$ are solutions to
\begin{align}\label{e5.1}
\begin{cases} 
\partial_t z(t,\xi) = \partial^2_{\xi\xi} ((a z)(t,\xi)) + z(t,\xi) 
\mu (\d t, \xi),\\
z(0,\;\cdot\;) = z_0.
\end{cases}
\end{align}
\end{enumerate}
Then $z^1 \equiv z^2$.
\end{theo}

\begin{rem} \label{R5.2}
By solution of equation \eqref{e5.1} we intend,
 as expected, the following: for every $\varphi \in \shs(\R), \ \forall t \in [0,T]$,
 \begin{equation}\label{e5.2}
\int_\R \varphi(\xi) z(t,\d \xi) = \left< z_0, \varphi \right>  + \int^t_0 \d s \int_\R a (s,\xi)\varphi''(\xi)
 z(s,\d \xi) 
+ \sum_{j=0}^N  \int_0^t \d W^j_s \int_\R \varphi(\xi) e^j(\xi) z(s,d\xi).
\end{equation}


\end{rem} 

\begin{proof}  [Proof of Theorem \ref{T51}]
The proof makes use of the similar  arguments as in
Theorem 3.8 of \cite{BRR1} or Theorem 3.1 in \cite{BCR2},
in a randomized form. The full proof is given in 
\cite{BRR4} Theorem 4.2,  see also \cite{Bar-Roc-Rus-2200}.

\end{proof}

\section{The non-degenerate
 case} \label{S6}

\setcounter{equation}{0}

We are now able to discuss the doubly probabilistic representation 
of a solution to  \eqref{PME} when $\psi$ is non-degenerate 
provided that its solution fulfills some properties.


 \begin{defi} \label{DNond} 
\begin{itemize} 
\item We will say that equation (\ref{PME}) (or $ \psi$) is 
{\bf non-degenerate} if on each compact,
 there is a constant $c_0 > 0$ such that
$ \Phi  \ge c_0  $.
\item We will say that equation (\ref{PME}) or $ \psi$ is 
{\bf degenerate} if  $ \lim_{u \rightarrow 0_+} \Phi(u) = 0  $.
\end{itemize} 
\end{defi}
 One of the typical examples of degenerate $\psi$ is the
case of $\psi$ being {\bf strictly increasing after some zero}.
This notion was introduced in \cite{BRR2} and it means the following.
There is $0 \le u_c $ such that $\psi_{[0,u_c]} \equiv 0$ and
$\psi$ is strictly increasing on $]u_c, +\infty[$.

\begin{rem}\label{DegNonD}
\begin{enumerate}
\item 
 $\psi$ is non-degenerate if and only if
$   = \lim_{u \rightarrow 0+} \Phi(u) > 0 $.
\item Of course, if $\psi$ is strictly increasing  after 
some zero, with $u_c > 0$ then
 $\psi$ is degenerate.
\item
 If  $\psi$ is degenerate, then
$\psi^\kappa (u) =  (\Phi^2(u) + \kappa) u $,
for every $\kappa > 0$, is non-degenerate.
 \end{enumerate}
\end{rem}
As announced the theorem below
also holds when $\psi$ is multi-valued.\\

\begin{theo}\label{thm6.1}
We suppose the following assumptions.
\begin{enumerate}
\item $x_0$ is a real probability measure. 
\item $\psi$ is non-degenerate.
\item There is only one  random field
 $X:[0,T]\times \R\times \Omega \to \R$ 
solution of \eqref{PME} (see Definition \ref{DSPDE})
such that 
\begin{equation}\label{e6.1}
\int\limits_{[0,T]\times \R} X^2 (s,\xi) \d s \d \xi < \9 \quad \text{a.s.}
\end{equation}
Then there is a unique weak  solution to the (DSNLD)$(\Phi,\mu,x_0)$.
\end{enumerate}
\end{theo}

\begin{rem} \label{R6.2}  
\begin{enumerate}  
\item 
 An easy adaptation of Theorem 3.4 of \cite{BRR3} (taking into account $e^0$),
 when $\psi$  is Lipschitz 
and $e^0, \ldots, e^N$ belong to $H^1$ allows to show 
that  there is a solution to \eqref{PME} such that\\
$E\( \int\limits_{[0,T]\times \R} X^2 (s,\xi) \d s \d \xi\) < \9.$
This holds even if $x_0$ belongs to $ H^{-1}(\R)$.
According to Theorem \ref{TB1}, that solution is unique. 
In particular item 3. in Theorem \ref{thm6.1} statement holds.
\item
 Theorem \ref{thm6.1} constitutes the converse of Theorem \ref{T32}
 when $\psi$ is non-degenerate.
\item
The theorem also holds if $\psi$ is multi-valued.
For implementing this, we need to adapt  the techniques of \cite{BRR1}.
\item
 As side-effect of the proof of the weak-strong existence Proposition 
\ref{P4.1}, the space $(\O_0,\shg,{\bf Q})$ can be chosen as 
$\O_0=\O_1\times \O,\; \O_1= C([0,T];\; \R) \times \R, \; \shg =\shb(\O_1)\times 
\shf,\; {\bf Q}(H\times F)=\int\limits_{\O_1\times\O} \d P(\o) 
1_F (\o)Q(\d \o_1,\o)$.
\end{enumerate}
\end{rem}

\begin{proof}
\begin{enumerate}
[1)]
\item We set $\gamma(t,\xi,\o)=\Phi \(X(t,\xi,\o)\right)$. According to
 Proposition \ref{P4.1} there is a weak-strong solution $Y$ to (DSDE)$(\gamma,x_0)$. By Proposition \ref{P4.2} $\o$ a.s. the $\mu$-marginal weighted 
laws of $Y$ admit  densities $\(q_t (\xi,\o),t\in]0,T],\; \xi \in\R,\; \o \in \O\right)$ such that $\pas$
$\int\limits_{[0,T]\times \R} \d s \d \xi q^2_s (\xi,\;\cdot\;) < \9 \quad \text{a.s.} $
\item Setting 
\begin{equation*}
\nu_t (\xi,\o)= \left(
\begin{array}{ccc}
q_t(\xi,\o) \d \xi &:& t\in ]0,T],\\
x_0 &:& t=0,
\end{array}
\right.
\end{equation*}
$\nu$ is a solution to \eqref{e5.1} with $\nu_0=x_0,\; a(t,\xi,\o)=
\Phi^2(X(t,\xi,\o))$. This can be shown applying It\^o's formula similarly as in the proof of Theorem \ref{T32}.
\item On the other hand $X$ is obviously also a solution of \eqref{e5.1},
which  in particular verifies \eqref{e6.1}.
Consequently $z^1=\nu, \; z^2=X$ verify items i), ii), iii) of
 Theorem \ref{T51}.
So Theorem \ref{T51} implies that $\nu \equiv X$; this shows that $Y$ provides a solution to  (DSNLD)$(\Phi,\mu,x_0)$.
\item Concerning uniqueness, let $Y^1,Y^2$ be two 
solutions to the (DSNLD) related to $(\Phi,\mu,x)$. The 
corresponding random fields $X^1,X^2$ constitute the $\mu$-marginal laws of $Y^1,Y^2$ respectively.
\end{enumerate}
Now $Y^i,\; i=1,2$, is a weak-strong solution of (DSDE)$(\gamma_i,x)$ with
 $\gamma_i(t,\xi,\o)=\Phi(X_i(t,\xi,\o))$, so by Proposition 
\ref{P4.2}
 $X_i,\; i=1,2$ fulfills 
\eqref{e6.1}. By Theorem \ref{T32}, $X_1$ and $X_2$ are solutions to
\eqref{PME}.
 By assumption 3. of the statement,
 $X_1=X_2$. The conclusion follows by Proposition \ref{P4.1}, which  
guarantees the uniqueness of the weak-strong solution of (DSDE)$(\gamma,x_0)$
 with $\gamma=\gamma_1=\gamma_2$.
\end{proof}

\begin{rem} \label{R6.3}
One side-effect of Theorem \ref{thm6.1} is the following. Suppose $\psi$
 to be non-degenerate. Let $X:[0,T]\times \R\times\O\to\R$ be a solution such that $\pas$ \\
$\int\limits_{[0,T]\times \R} X^2 (s,\xi) \d s \d \xi <\9 \quad \text{a.s.}$
We have the following for $\o \pas$
\begin{enumerate}[i)]
\item $X(t,\; \cdot\;,\o)\geq 0$ a.e. $\forall\; t\in [0,T].$
\item $E\(\int\limits_\R X(t,\xi)\d \xi\right)=1,\; \forall \; t\in [0,T]$ if $e_0=0$.
\end{enumerate}
\end{rem}
\begin{rem}\label{R6.4}
If \eqref{PME} has a solution, not necessarily unique, then (DSNLD) 
with respect to $(\Phi,\mu,x_0)$ still admits existence.
\end{rem}

\section{The degenerate case}
\label{S7}

\setcounter{equation}{0}

The idea consists in proceeding similarly to \cite{BRR2}, which treated
 the case $\mu=0$ and the case when $x_0$ is absolutely continuous
with bounded density.
$\psi$ will be assumed to be strictly increasing after some zero $u_c \ge 0$,
see Definition \ref{DNond}.
We recall that if $\psi$ is degenerate, then necessarily
$\Phi(0):=  \lim_{x \rightarrow 0} \Phi(x) = 0$.

The theorem below concerns existence, we do not know any uniqueness result
in the degenerate  case.
\begin{theo} \label{T73}
We suppose the following.
\begin{enumerate}
\item The functions $ e^i, 1 \le i \le N$ 
 belong to $H^1(\R)$. 
\item We suppose that $\psi:\R\to\R$ is non-decreasing, Lipschitz and 
strictly increasing after some zero.
\item $x_0$ belongs to $L^1(\R) \cap L^2(\R)$.
\end{enumerate}
 Then there is a weak solution to the (DSNLD)$(\Phi, \mu, x_0)$.
\end{theo}
\begin{rem} \label{R7.2}
If $u_c>0$ then $\psi$ is necessarily degenerate
and also $\Phi$ restricted to $[0,u_c]$ vanishes. 
\end{rem}
\begin{prooff}  \ (of Theorem \ref{T73}).
\begin{enumerate}[1)]
\item We proceed by approximation rendering $\Phi$ non-degenerate. Let $\kappa>0$. We define $\Phi_\kappa:\R\to\R_+$ by
$ \Phi_\kappa (u) = \sqrt{\Phi^2(u)+\kappa}, \quad
\psi_\kappa (u) = \Phi_\kappa^2(u)\cdot u. $
Let $X^\kappa$ be the solution so \eqref{PME} with $\psi_\kappa$ instead of $\psi$.
According to Theorem \ref{thm6.1} and Remark \ref{R6.2} 4.,
 setting 
\begin{equation} \label{ENummer}
\wt \O_1=C\([0,T],\R\)\times \R, \quad  Y(\o_1,\o)=\o_1,
\end{equation}
$\shh$  the Borel $\sigma$-algebra of $\wt \Omega_1,$
there  are families of probability kernels $ Q^\kappa$ on $\shh\times\O,$
and measurable processes $B^\kappa$ on $\wt \O_0 = \wt \Omega_1 \times \Omega$ such that
\begin{enumerate}[i)]
\item $B^\kappa (\;\cdot\;,\o)$ is a
 $Q^\kappa (\;\cdot\;,\o)$-Brownian motion;
\item \label{e7.3} $Y$ is a (weak) solution, on $(\tilde \Omega_1, Q^\kappa(\cdot,\o))$, of\\ 
$Y_t=Y_0 +\int\limits_0^t \Phi_\kappa (X^\kappa(s,Y_s,\o))\d B^\kappa_s(\cdot,\omega), \ t \in [0,T]; $
\item $Y_0$ is distributed according to $x_0=X^\kappa(0,\;\cdot\;)$.
\item The $\mu$-marginal weighted laws of $Y$ under ${\bf Q}^\kappa$ are $(X^\kappa(t,\;\cdot\;))$.
\end{enumerate}
In agreement with Definition \ref{DDoubleStoch} and Definition \ref{DWeakStrong}, 
we need to show the existence  of a suitable measurable space $(\Omega_1,\shh)$, 
a probability kernel $Q$ on 
 $\shh\times\O$, a process $B$ on $\O_0:= \Omega_1 \times \Omega$ such that 
the following holds.
\begin{enumerate}[i)]
\item $B (\;\cdot\;,\o)$ is a $Q (\;\cdot\;,\o)$-Brownian motion.
\item $Y$ is a (weak) solution on $(\Omega_1, Q(\cdot,\o))$ of \\
$ Y_t=Y_0+\int\limits_0^t \Phi (X(s,Y_s,\o))\d B_s(\cdot,\omega), \ t \in[0,T], $
i.e. item 1) of Definition \ref{DWeakStrong}. 
Moreover items 2), 3) of the same Definition are fulfilled.
\item $Y_0$ is distributed according to $x_0$.  
\item For every $t\in]0,T],\; \varphi \in C_b(\R)$, if we denote $Q^\o = Q(\;\cdot\;,\o)$, we have
\begin{equation}\label{e7.5}
\int_\R X(t,\xi) \varphi(\xi)\d \xi = E^{Q^\o} \left(\varphi(Y_t)
\she_t\(\int\limits_0^\cdot \mu (\d s, Y_s) X(s,Y_s)\)\right). 
\end{equation}
\end{enumerate}
\item We  show now that $X^\kappa$ approaches $X$ in some sense
 when $\kappa \to 0$, where $X$ is
 the solution to \eqref{PME}. This is given in the Lemma \ref{L7.4} below.
\end{enumerate}

\begin{lemma} \label{L7.4}
Under the  assumptions of Theorem \ref{T73}, 
according to Remark \ref{remB2bis},
let $X$ (resp. $X^\kappa$) be a solution of \eqref{PME}
verifying \eqref{equater} with $\psi(u) = u \Phi^2(u)$ (resp. $\psi_\kappa(u)
 = u(\Phi^2(u) + \kappa)$), for $ u > 0$.
We have the following.
\begin{enumerate}[a)]
\item
$
\lim_{\kappa \to 0} \sup_{t\in [0,T]} E\(\norm{X^\kappa (t,\;\cdot\;)-X(t,\;\cdot\;)}^2_{H^{-1}}\)=0;
$
\item
$\lim_{\kappa \to 0}E\(\int^T_0 \d t \norm{\psi\(X^\kappa (t,\;\cdot\;)\)-\psi\(X(t,\;\cdot\;)\)}^2_{L^2}\)=0;$
\item  $\lim_{\kappa \to 0} \kappa  E\(\int_{[0,T]\times \R}
 \d t \d \xi \(X^\kappa (t,\xi) - X(t,\xi)\)^2 \) = 0. $
\end{enumerate}
\end{lemma}
\begin{rem} \label{R7.5}
\begin{enumerate}[1)]
\item a) implies of course \\
$\lim_{\kappa \to 0}E\(\int^T_0 \d t \norm{X^\kappa (t,\;\cdot\;)-X(t,\;\cdot\;)}^2_{H^{-1}}\)=0.$
\item In particular Lemma \ref{L7.4} b) implies that for each sequence $(\kappa_n)\to 0$ there is a subsequence, still denoted by the same notation, that
\\ 
$\int\limits_{[0,T]\times\R} \(\psi (X^{\kappa_n} (t,\xi))-\psi(X(t,\xi))\)^2 \d t \d \xi \substack{\longrightarrow\\n\to\9}0 \ {\rm a.s.}$
\item 
For every $t\in[0,T]$
$X(t,\;\cdot\;) \geq 0 \quad  d\xi \otimes dP \text{a.e.}$
Indeed, for this it will be enough to show that a.s.
\begin{equation}\label{e7.6}
\int\limits_\R \d \xi \varphi(\xi)X(t,\xi) \ge 0 \text{ for every }
 \varphi \in \shs(\R),
\end{equation}
for every $t \in [0,T]$.
Since $X\in C\([0,T];\;\shs'(\R)\)$ it will be enough to show \eqref{e7.6} for almost all $t\in[0,T]$. 
This holds true since item 1) in this Remark
\ref{R7.5}, implies the existence of a sequence $(\kappa_n)$ such that \\
$\int\limits^T_0 \d t\norm{X^{\kappa_n}(t,\cdot) - X(t,\;\cdot\;)}^2_{H^{-1}}
\substack{\longrightarrow\\n\to\9}0, \quad {\rm a.s.} $
\item Since $\psi$ is strictly increasing after   $u_c$, 
then, for $P$ almost all $\o$, for almost all $(t,\xi)\in [0,T]\times\R$,
 there is a sequence $(\kappa_n)$ such that 
$ \(X^{\kappa_n}(t,\xi)-X(t,\xi)\) 1_{\{X(t,\xi)>u_c\}}
\substack{\longrightarrow\\n\to\9}0. $
\\
This follows from item 2) of Remark \ref{R7.5}.\\
Since $\Phi^2(u)=0$ for $0\leq u\leq u_c$ and $X$ is a.e. non-negative, 
this implies that $dt d\xi dP$ a.e. we have
\begin{equation}\label{e7.8}
\Phi^2\(X(t,\xi)\)\(X^{\kappa_n}(t,\xi)-X(t,\xi)\)\substack{\longrightarrow\\n\to\9}0.
\end{equation}
\end{enumerate}
\end{rem}
\begin{proof}[Proof (of Lemma \ref{L7.4})]
By Remark \ref{remB2bis} 3.
 we can write $\pas$ the following  $H^{-1}(\R)$-valued
 equality.
$$
\(X^\kappa -X\)(t,\;\cdot\;)= \int\limits^t_0 \d s\( \psi_\kappa \(X^\kappa(s,\;\cdot\;)\)-\psi \(X(s,\;\cdot\;)\)\)''
 + \sum^N_{i=0} \int\limits^t_0  
\(X^\kappa(s,\;\cdot\;)-X(s,\;\cdot\;)\) e^i \d W_s^i.
$$
So
\begin{align*}
(I-\Delta)^{-1} \(X^\kappa -X\)(t,\;\cdot\;)= & -\int\limits^t_0 \d s\( \psi_\kappa \(X^\kappa(s,\;\cdot\;)\)-\psi \(X(s,\;\cdot\;)\)\)\\
& + \int\limits^t_0 \d s (I-\Delta)^{-1} \( \psi_\kappa \(X^\kappa(s,\;\cdot\;)\)-\psi \(X(s,\;\cdot\;)\)\)\\
& + \sum^N_{i=0} \int\limits^t_0 (I-\Delta)^{-1} \( e^i \(X^\kappa(s,\;\cdot\;)-X(s,\;\cdot\;)\)\)\d W_s^i. 
\end{align*}
After regularization and application of It\^o calculus with values in $H^{-1}$, we will be able to estimate
 $g^\kappa(t)=\norm{\(X^\kappa - X\)(t,\;\cdot\;)}^2_{H^{-1}}$.
Taking advantage
of the form of $\psi_\kappa - \psi  $, 
we obtain
\begin{align}\label{e7.9}
g^\kappa(t) = & \sum^N_{i=1} \int\limits^t_0 \norm{e^i \(X^\kappa -X\)(s,\;\cdot\;)}^2_{H^{-1}} \d s\\
- & 2 \int\limits^t_0 \<\(X^\kappa -X\)(s,\;\cdot\;), \psi_\kappa \(X^\kappa(s,\;\cdot\;)\)-\psi \(X(s,\;\cdot\;)\)\>_{L^2}\\
+ & 2 \int\limits^t_0 \d s \< \(X^\kappa -X\)(s,\;\cdot\;),(I-\Delta)^{-1}
\left(\psi_\kappa \(X^\kappa(s,\;\cdot\;)\)-\psi \(X(s,\;\cdot\;)\) \right)
\>_{L^2}\\
+ & 2 \int\limits^t_0 \d s \< \(X^\kappa -X\)(s,\;\cdot\;),(I-\Delta)^{-1} e^0 \(X^\kappa -X\)(s,\;\cdot\;)\>_{L^2}  + M^\kappa_t,
\end{align}
where $M^\kappa$ is the local martingale
\begin{equation*}
M^\kappa_t = 2 \sum^N_{i=1} \int\limits^t_0 \< (I-\Delta)^{-1} \(X^\kappa -X\)(s,\;\cdot\;),\(X^\kappa -X\)(s,\;\cdot\;) e^i\>_{L^2} \d W^i_s.
\end{equation*}
Indeed, $M^\kappa$ is a well-defined local martingale because, taking into account
\eqref{eFB1} and Remark \ref{remB2bis}, using classical arguments, 
we can prove that
$$  \sum^N_{i=1} \int\limits^t_0 \vert \< \(  X^\kappa  -  X \vert\)
(s,\;\cdot\;),(I-\Delta)^{-1}\(X^\kappa -X\)(s,\;\cdot\;)
 e^i\>_{L^2} \vert^2 \d s < \infty \ {\rm a.s.}  $$
\eqref{e7.9} gives
\begin{align}
g^\kappa(t) & + 2\int\limits^t_0 \<\(X^\kappa-X\)(s,\;\cdot\;),\;\psi \(X^\kappa(s,\;\cdot\;)\)-\psi \(X(s,\;\cdot\;)\)\>_{L^2} \d s\\
& + 2\kappa \int\limits^t_0 \<\(X^\kappa-X\)(s,\;\cdot\;),\;\(X^\kappa-X\)(s,\;\cdot\;)\>_{L^2} \d s\\
  \leq & -2\kappa \int\limits^t_0 \d s\<\(X^\kappa-X\)(s,\;\cdot\;),\; X(s,\;\cdot\;)\>_{L^2} \d s
 + \sum^N_{i=1}\int\limits^t_0 \norm{e^i \(X^\kappa -X\)(s,\;\cdot\;)}^2_{H^{-1}} \d s\\
& + 2 \int\limits^t_0 \d s\<(I-\Delta)^{-1}\(X^\kappa -X\)(s,\;\cdot\;),\;\psi \(X^\kappa(s,\;\cdot\;)\)-\psi \(X(s,\;\cdot\;)\)\>_{L^2} 
\end{align}
\begin{align}
& + 2\kappa \int\limits^t_0 \d s\<(I-\Delta)^{-1}\(X^\kappa -X\)(s,\;\cdot\;),\;\(X^\kappa-X\)(s,\;\cdot\;)\>_{L^2}\\
& + 2\kappa \int\limits^t_0 \d s\<(I-\Delta)^{-1}\(X^\kappa -X\)(s,\;\cdot\;),\; X(s,\;\cdot\;)\>_{L^2}\\
& + 2 \int\limits^t_0 \d s\< (I-\Delta)^{-1} \(X^\kappa -X\)(s,\;\cdot\;),\;
\(e^0\(X^\kappa -X\)(s,\;\cdot\;)\)\>_{L^2}  + M^\kappa_t .
\end{align}
We use Cauchy-Schwarz and the inequality
$2\sqrt{\kappa} b \sqrt{\kappa}c\leq \kappa b^2 + \kappa c^2, $
with first \\
$
b  = \norm{X^\kappa(s,\;\cdot\;)-X(s,\;\cdot\;)}_{L^2}, \quad
c  = \norm{X(s,\;\cdot\;)}_{L^2}
$
and then \\
$
b  = \norm{X^\kappa(s, \cdot) -X (s,\cdot)}_{H^{-2}}, \quad
c  = \norm{X(s,\;\cdot\;)}_{L^2}.
$
We also take into account the property of $H^{-1}$-multiplier 
for $e^i,\; 0\leq i\leq N$. Consequently there is a constant
 ${\mathcal C}(e)$ depending on $(e^i,\; 0\leq i\leq N)$ such that
\begin{align}\label{e7.11}
g^\kappa(t) & + 2\int\limits^t_0 \<\(X^\kappa-X\)(s,\;\cdot\;),\;\psi \(X^\kappa(s,\;\cdot\;)\)-\psi \(X(s,\;\cdot\;)\)\>_{L^2} \d s\\
& + 2\kappa \int\limits^t_0 \norm{X^\kappa(s,\;\cdot\;)-X(s,\;\cdot\;)}^2_{L^2} \d s \\
& \leq  \kappa \int\limits^t_0 \norm{\(X^\kappa -X\)(s,\;\cdot\;)}^2_{L^2} \d s
 + \kappa \int\limits^t_0 \d s \norm{X(s,\;\cdot\;)}_{L^2}^2\\
& + {\mathrm C}(e)\int\limits^t_0 \d s \norm{X^\kappa(s,\;\cdot\;)-X(s,\;\cdot\;)}^2_{H^{-1}}\\ 
& + 2\int\limits^t_0 \norm{\(X^\kappa -X\)(s,\;\cdot\;)}_{H^{-2}} \norm{\psi \(X^\kappa(s,\;\cdot\;)\)-\psi \(X(s,\;\cdot\;)\)}_{L^2}\\
& + 2\kappa \int\limits^t_0 \d s g^\kappa(s)
 + \kappa \int\limits^t_0 \d s \norm{\(X^\kappa -X\) (s,\;\cdot\;)}_{H^{-2}}^2 
+ \kappa \int\limits^t_0 \d s \norm{X(s,\;\cdot\;)}_{L^2}^2
 + M^\kappa_t .
\end{align}
Since $\psi$ is Lipschitz, it follows
$ \(\psi(r)-\psi(r_1)\)(r-r_1) \geq \alpha\(\psi(r)-\psi(r_1)\)^2, $
for any $r, r_1 \ge 0$,
for some $\alpha>0$. Consequently, the inequality
$ 2bc \leq b^2 \alpha+\frac{c^2}{\alpha},$
with $b,c\in\R$ and the fact that $\norm{\;\cdot\;}_{H^{-2}}\leq \norm{\;\cdot\;}_{H^{-1}}$ give
\begin{align}\label{e7.12}
2 & \int\limits^t_0 \d s \norm{\(X^\kappa -X\)(s,\;\cdot\;)}_{H^{-2}} \norm{\psi \(X^\kappa(s,\;\cdot\;)\)-\psi \(X(s,\;\cdot\;)\)}_{L^2}\\
\leq & \int\limits^t_0 \d s \alpha g^\kappa (s,\;\cdot\;)+ \int\limits^t_0
 \d s\<\psi \(X^\kappa(s,\;\cdot\;)\)-\psi \(X(s,\;\cdot\;)\),\;X^\kappa(s,\;\cdot\;)-X(s,\;\cdot\;)\>_{L^2}.
\end{align}
So \eqref{e7.11} yields 
\begin{align}\label{e7.13}
g^\kappa(t) & + \int\limits^t_0 \< X^\kappa(s,\;\cdot\;)-X(s,\;\cdot\;),\;\psi \(X^\kappa(s,\;\cdot\;)\)-\psi \(X(s,\;\cdot\;)\)\>_{L^2} \d s\\
& + \kappa \int\limits^t_0 \d s \norm{X^\kappa(s,\;\cdot\;)-X(s,\;\cdot\;)}^2_{L^2} \d s\\
& \leq  2\kappa \int\limits^t_0 \d s \norm{X(s,\;\cdot\;)}^2_{L^2}
 + M^\kappa_t 
 + \({\mathrm C}(e) +\alpha+3\kappa\)\int\limits^t_0 g^\kappa (s)\d s.
\end{align}
Taking the expectation we get
$$ E(g^\kappa(t)) \le  \({\mathrm C}(e) + \alpha + 3\kappa\) \int\limits^t_0 E(g^\kappa (s)) \d s \\
+ 2 \kappa \int_0^t E(\Vert X(s,\cdot) \Vert^2_{L^2}) \d s, $$
for every $t\in[0,T]$.
By Gronwall lemma we get
\begin{equation} \label{e7.14bis}
E\(g_\kappa(t)\) \leq 2 \kappa E \left\{ \int\limits^T_0 \d s
\norm{X(s,\;\cdot\;)}^2_{L^2}\right\}e^{\({\mathrm C}(e) +\alpha+3\kappa\)T}, 
\quad \forall \; t\in [0,T].
\end{equation}
Taking the supremum and letting $\kappa \to 0$, item a) of Lemma \ref{L7.4}
 is now established. \\
We go on with item b). Since $\psi$ is Lipschitz, \eqref{e7.13} implies that,
 for $t\in [0,T]$,
\begin{align*}
& \int\limits^t_0 \d s\norm{\psi \(X^\kappa(s,\;\cdot\;)\)-\psi \(X(s,\;\cdot\;)\)}^2_{L^2}\\
\leq & \frac{1}{\alpha} \d s \<\psi \(X^\kappa(s,\;\cdot\;)\)-\psi \(X(s,\;\cdot\;)\),\; X^{(\kappa)}(s,\;\cdot\;)-X(s,\;\cdot\;)\>_{L^2}\\
\leq & \frac{\kappa}{2\alpha} \int\limits^t_0 \d s \norm{X(s,\;\cdot\;)}^2_{L^2}
 + {\mathrm C}(e,\alpha)  \int\limits^t_0 g_\kappa(s) \d s + M^\kappa_t,
\end{align*}
where ${\mathrm C}(e,\alpha)$ is a constant depending on
 $e^i, 0\leq i\leq N$ and $\alpha$.
Taking the expectation for $t=T$, we get
$$
 E\(\int\limits^T_0 \d s\norm{\psi \(X^\kappa(s,\;\cdot\;)\)-\psi \(X(s,\;\cdot\;)\)}_{L^2}^2\)
 \leq  \frac{\kappa}{2 \alpha}  E\(\int\limits^T_0 \d s\norm{X(s,\;\cdot\;)}^2_{L^2}\)+ {\mathrm C}(e,\alpha) \int\limits^T_0 E(g_\kappa(s))\d s.
$$
Taking $\kappa \to 0$, 
 \eqref{equater}  and \eqref{e7.14bis} provide the conclusion of
 item b) of Lemma \ref{L7.4}.
\begin{enumerate}[c)]
\item Coming back to \eqref{e7.13}, and $t=T$, we have 
$$\kappa \int\limits^T_0 \d s \norm{X^\kappa(s,\;\cdot\;)-X(s,\;\cdot\;)}_{L^2}^2
\leq
 2 \kappa \int\limits^T_0 \d s \norm{X(s,\;\cdot\;)}_{L^2}^2 + M^\kappa_T 
 + \({\mathrm C}(e) +\alpha+3\kappa\) \int\limits^T_0 \d s g^\kappa(s).
$$
Taking the expectation we have
$$
\kappa E\(\int\limits^T_0 \d s\norm{X^\kappa(s,\;\cdot\;)-X(s,\;\cdot\;)}_{L^2}^2\) 
\leq 
 2\kappa E \left( \int\limits^T_0 \d s\norm{X(s,\;\cdot\;)}^2_{L^2}\right)
 + \({\mathrm C}(e) +\alpha+3\kappa\) E\(\int\limits^T_0 g^\kappa(s) \d s\).
$$
Using item a) and the fact that
$E\(\int\limits_{[0,T]\times\R} X^2 (s,\xi)\d s \d \xi\)<\9, $
the result follows.
Lemma \ref{L7.4} is finally completely established.
\end{enumerate}
\end{proof}
We need now another intermediate lemma concerning the paths of a solution to \eqref{PME}.
\begin{lemma} \label{L7.6}
For almost all $\o\in\O$, almost all $t\in[0,T]$,
\begin{enumerate}[1)]
\item $\xi \mapsto \psi(X(t,\xi,\o))\in H^1(\R)$,
\item $\xi \mapsto \Phi(X(t,\xi,\o))$ is continuous.
\end{enumerate}
\end{lemma}
\begin{proof}[Proof]
Item 1) is established in \cite{BRR3}, see Definition 
3.2 and Theorem 3.4. 1) implies that $\xi 
\mapsto \psi(X(t,\xi,\o))$ is continuous.
See also Remark \ref{remB2bis} 1. 
By the same arguments as in Proposition 4.22 in \cite{BRR2}, we can deduce item 2).
\end{proof}
\begin{description}
\item{3)} We go on with the proof of Theorem \ref{T73}. 
We keep in mind i), ii), iii), iv) at the beginning of item 1)
of the proof.
 Since $\Phi$ is bounded,
for $P$-almost all $\omega$, 
 using Burkholder-Davis-Gundy inequality one obtains
\begin{equation} \label{TE3}
E^{Q^{\kappa}(\;\cdot\;,\o)}\(Y_t-Y_s\)^4 \leq \const(t-s)^2,
\end{equation}
where ${\rm const}$ does not depend on $\omega$.
On the other hand, for all 
$Q^\kappa (\;\cdot\;,\o), \; Y_0$ is distributed according to $x_0$.

At this point, we need a version of Kolmogorov-Centsov theorem for the
stable convergence. Let $\tilde \O_0 = \tilde \Omega_1 \times \Omega$
 as at the beginning of
the proof of Theorem \ref{T73}. We recall that 
$\tilde \Omega_1 = C([0,T]) \times \R$,
   $ Y(\omega_1,\omega) = \omega_1$, $ \shh$ is the Borel $\sigma$-field
on $\tilde \Omega_1$.

\begin{lemma} \label{TStableKolm}
Let be a sequence $Q^\kappa(\cdot,\omega)$ of random kernel on 
$\shh \times \Omega$. Let us denote by ${\bf Q}^\kappa$ the sequence
 of marginal laws of the probabilities 
 on $(\tilde \O_0,\shh \otimes \shf)$ given by $ Q^{\kappa}(\cdot, \omega)  P(\d \o) $.
Suppose the following.
\begin{itemize}
\item The sequences of marginal laws of the probabilities ${\bf Q}^\kappa$ 
at zero are tight.
\item There are $\alpha, \beta > 0$ such that
\begin{equation} \label{EStableKolm}
 E^{Q^\kappa(\cdot, \omega)} \vert Y_t - Y_s\vert^\alpha \le C(\o)(t-s)^{1+\beta},
\quad 0 \le s \le T,
\end{equation}
for some positive $P$-integrable random constant $C$.
\end{itemize}
Then there is  a random kernel $Q^\infty$ on $\shh \times \Omega$
 and a subsequence $(\kappa_n)$ such that
for every bounded continuous functional $G: \tilde \Omega_1 \rightarrow \R$,
for every bounded $\shf$-measurable r.v. $F: \Omega \rightarrow \R$,
we have
\begin{equation} \label{E1}
\begin{array}{ccc}
\int_\Omega F(\omega) \d P(\omega)&& \int_{\tilde \Omega_1} G(Y(\omega_1))
Q^{\kappa_n}(\d \omega_1, \omega) 
\rightarrow_{n \rightarrow \infty}  
\\
 &&\int_\Omega F(\omega) \d P(\omega)
 \int_{\tilde \Omega_1} G(Y(\omega_1)) Q^\infty(\d \omega_1,\omega). 
\end{array}
\end{equation}

\end{lemma}
\begin{proof}

Taking the expectation with respect to $P$ we obtain
 \begin{equation} \label{EStableKolmbis}
 E^{{\bf Q} ^\kappa} (Y_t - Y_s)^\alpha \le C_0 (t-s)^{1+\beta},
\quad 0 \le s \le T,
\end{equation}
where $C_0$ is the expectation of $C$.
First, by usual arguments as Chebyshev inequality, one can show 
the following: 
\begin{eqnarray*} 
\lim_{\lambda \rightarrow \infty} \sup_{\kappa}
 {\bf Q}^{\kappa}\{(\o_1,\omega) \vert \vert (W^1,\ldots,W^N)(\o)(0)\vert > \lambda; 
\vert \omega_1(0)\vert
 > \lambda \} &=& 0, \\
\lim_{\delta \rightarrow 0} \sup_{\kappa}
    {\bf Q}^{\kappa} \{ (\o_1, \omega) 
\vert m((W^1,\ldots,W^N, \omega_1);
\delta) > \varepsilon \} &=& 0,
\forall \varepsilon > 0,
\end{eqnarray*}
where $m$ denotes the modulus of continuity.
By Theorem 4.10 of \cite{KARSH},  the sequences of probabilities
 $ {\bf Q}^{\kappa}, \kappa >0$, on $\tilde \Omega_1 \times \Omega$ are tight. 
Let ${\bf Q}^{\kappa_n}$ be  a sequence converging weakly to a probability
${\bf Q}^\infty$ on $\shh \otimes \shf$.
Since $\shf$ is separable and $C([0,T])^N$, which is space value
of process $W$, is a Polish space equipped with its Borel $\sigma$-algebra,
according to \cite{ripley},
 it is possible to desintegrate ${\bf Q}^\infty$,
i.e. there is random kernel $Q^\infty(\cdot, \omega)$ such that 
for every bounded continuous functional $G: \tilde \Omega_1 \rightarrow \R$,
for every bounded  continuous  $\tilde F: C([0,T])^N \rightarrow \R$
such that $\eqref{E1}$ holds for every $F = \tilde F(W)$,
where $W= (W^1,\ldots, W^N)$.
Since continuous bounded functionals $\tilde F$ are dense in
$L^2(C([0,T])^N$ equipped with Wiener measure, (\ref{E1}) 
holds also for any $F$ bounded $\shf$-measurable r.v.
with ${\bf Q}^\infty(\d \o_1, \d \o)  = Q^\infty(\d \o_1, \omega)  P(\d \o)$.
\end{proof}

By \eqref{TE3}, 
we apply Lemma \ref{TStableKolm} with $\alpha = 2, \beta = 1$
  and we consider the
corresponding $Q^{\kappa_n}(\cdot,\omega)$ and the limit
random kernel  $Q(\cdot, \omega) := Q^{\infty} (\cdot,\omega)$.
We define also the probability ${\bf Q} := {\bf Q}^\infty$  on
 $\tilde \Omega_0 = 
\tilde \Omega_1 \times \Omega$
according to the conventions introduced before Definition \ref{DSEPS}.
In the sequel we denote again by $\d {\bf Q}^\kappa(\o_1, \o):= \d P (\o)
 Q^\kappa(\d \o_1,  \o)$
and also  $Q^{\o,\kappa}:= Q^\kappa(\cdot, \o)$,
$Q^{\o}:= Q(\cdot, \o)$.




From Lemma \ref{TStableKolm} derives the following.
\begin{cor} \label{C7.6}
For  any bounded random element $F: \tilde \Omega_1 \times \Omega 
 \rightarrow \R$ such that
for almost all $\o \in \Omega$, $F(\cdot, \o) \in C(\tilde \Omega_1)$.
Then \\
$ \int_\Omega \d P(\omega) \int_{\tilde \Omega_1} 
  \left(\d Q^{\o,\kappa_n}(\o_1) - \d Q^{\o}(\o_1)\right)  F(Y,\o)$
 converges to zero.
\end{cor}
\begin{proof} \  See Appendix A. \end{proof}



\end{description}
We need here a technical lemma.
\begin{lemma} \label{L7.6bis}
Let $t \in [0,T],$  $p \in \R$.
\begin{enumerate}
\item 
There is $C(p) > 0$ such that
 \begin{equation} \label{e7.24ter}
 E^{{\bf Q}^\kappa} \(\she_t \( \int\limits^\cdot_0 \mu (\d s,Y_s)\)^p\) \le C(p), \ \forall
 \kappa > 0.
\end{equation}
\item 
 For almost all $\o \in \Omega$, and every $p \in \R$ 
 there is a random constant $C(p,\omega)$ 
 such that
  the random variables
 \begin{equation} \label{e7.24bis}
  E^{Q^{\o, \kappa}} \(\she_t \( \int\limits^\cdot_0 \mu (\d s,Y_s)\)^p\)
 \le C(p,\omega), \
\forall \kappa > 0.
 \end{equation}
\end{enumerate}
 \end{lemma}
\begin{proof}
\ Without restriction of generality we can of course suppose $e^0=0$.
\begin{enumerate} 
\item
 We can write 
\begin{eqnarray*} 
\she_t \( \int\limits^\cdot_0 \mu (\d s,Y_s)\)^p &=& 
\she_t \( p \int\limits^\cdot_0 \mu (\d s,Y_s)\) 
\exp \left(\frac{p^2 - p}{2} \sum_{i=1}^N \(\int\limits^t_0 e^i 
 (Y_s)^2 \d s\)\right)  \\ & \le &  \she_t \( p \int\limits^\cdot_0 \mu (\d s,Y_s)\) 
 \exp\(T \frac{p^2-p}{2}  \sum_{i=1}^N \Vert e^i \Vert^2_\infty\).
\end{eqnarray*}
Since  $p \int\limits^t_0 \mu (\d s,Y_s)$  is a
 $(\shg_t)$-{\bf Q}$^\kappa$-martingale,
the result follows.
\item Let $\o \in \O$ excepted on a $P$-null set.
 The integrand  of the expectation in \eqref{e7.24bis}
equals  $ \exp \( J_1(n)+J_2(n)\),$
 where 
 $$
 J_1(n)  :=  p\sum^N_{i=1} \( W^i_t e^i (Y_t)-\frac{1}{2} \int\limits^t_0 e^i 
 (Y_s)^2 \d s
  - \frac{1}{2} \int\limits^t_0 W^i_s (e^i)'' (Y_s) \Phi^2 \(X (s,Y_s,\o)\) \d s\) 
 $$
 and 
  $J_2(n) = - p\sum_{i=1}^N \int\limits^t_0 W^i_s (e^i)' (Y_s)\d Y_s$. 
 For each $\o,\; \exp \( (J_1(n)\) $ is bounded, so
 it remains to prove the existence of
 a random constant $C(p,\o)$ such that
  for every $0\leq i\leq N$
 \begin{equation}\label{e7.24quater}
  E^{Q^{\o, \kappa}} \(\exp \(- p\int\limits^t_0 W^i_s (e^i)' (Y_s)\d Y_s\)\) 
 \le C(p,\o). 
\end{equation}

 Since $-p\int\limits^t_0 W^i_s (e^i)' (Y_s)\d Y_s $ is a 
 $Q^{\o,\kappa}$-martingale,
 \begin{equation*}
 \she_t^\kappa:=  \exp\(- p\int\limits^t_0 W^i_s (e^i)' (Y_s)\d Y_s
  -  \frac{p^2}{2} \int\limits^t_0 (W^{i})^2_s (e^i)^{'2} (Y_s)
 \Phi^2_{\kappa}\(X^{\kappa}(s,Y_s,\o)\)\d s\)
 \end{equation*}
 is an (exponential) martingale, with respect to $Q^{\o, \kappa}$.
 Consequently the left-hand side of \eqref{e7.24quater} is bounded by
 \begin{align*}
 & E^{Q^{\o, \kappa}} \(\she^\kappa_t
 \exp \( \frac{p^2}{2}\int\limits^t_0 (W^{i})^2_s ((e^{i})')^{2}
  (Y_s)
  \Phi^2_{\kappa}(X^{\kappa}(s,Y_s,\o))\d s\) \)\\
 \leq & C(p,\cdot):= 
 \exp \(\frac{p^2}{2} \norm{(e^{i})'}^2_\9 \(\norm{\Phi}^2_\9 + 1\) 
 \int\limits^T_0 (W^{i}_s)^2 \d s\).
 \end{align*}
 This concludes the proof of Lemma \ref{L7.6bis}.
\end{enumerate}
\end{proof}

\begin{lemma} \label{L7.7} We fix $\o \in \O$ excepted on some $P$-null set.
Let $\varphi: [0,T] \times \R \rightarrow \R$
continuous with compact support.
The random variables
 \begin{equation}\label{e7.16bis}
 E^{Q^{\o, \kappa}}\left( \int^T_0  \left \vert \Phi_{\kappa}
 \(X^{\kappa}(r,Y_r,\o)\)-
 \Phi\(X (r,Y_r,\o)\) \right \vert  \varphi(r,Y_r) \d r \right)
 \end{equation}
converge to zero a.s. and in  $L^p(\Omega, P)$ for every $p \ge 1$,
when $\kappa \rightarrow 0$.
\end{lemma}

\begin{proof}[Proof]

Let $\o \in \O$.
Since $\varphi$ has compact support, by Cauchy-Schwarz with respect 
to the measure  $\varphi(r,Y(r)) \d r$ on $[0,T]$, it is enough to prove 
that  
 \begin{equation}\label{e7.16}
 E^{Q^{\o, \kappa}}\left( \int^T_0 \(\Phi_{\kappa}
 \(X^{\kappa}(r,Y_r,\o)\)-
 \Phi\(X (r,Y_r,\o)\)\)^2 \varphi(r,Y_r) \d r \right)
 \end{equation}
converges to zero.
Since $\Phi$ is bounded it is enough to prove
the  convergence to zero for almost all $\o \in \O$.
 In order not to overcharge the notation, in this proof we will omit 
the argument of $\omega$ of $Y$.
By Fubini's theorem the left-hand side of \eqref{e7.16}
equals
 \begin{equation}\label{e7.16ter}
\int^T_0  \d r  E^{Q^{\o, \kappa}} \( (\Phi_{\kappa}   
   (X^{\kappa}(r,Y_r,\o))-
 \Phi(X (r,Y_r,\o)))^2 \varphi(r,Y_r)\).
 \end{equation}

Using also Lebesgue dominated convergence theorem,
given a sequence $(\kappa_n)$, when $n \rightarrow \infty$,
it is enough to find  a subsequence $(\kappa_{n_\ell})$  
such that for all $r \in [0,T]$ outside a possible Lebesgue null set
$$ E^{Q^{\o, \kappa_{n_\ell}}} \left \{ \( \Phi_{\kappa_{n_\ell}}   
   \(X^{\kappa_{n_\ell}}(r,Y_r,\o)\)-
 \Phi\(X (r,Y_r,\o)\)\)^2 \varphi(r,Y_r) \right\} 
\rightarrow_{\ell \rightarrow \infty} 0. $$
 We set $Z_r(\omega_1,\omega) =  \she_r \(\int\limits^\cdot_0 
\mu(\omega) (\d s,Y_s(\omega_1))\).$
We will  substitute from now on $(n_\ell)$ with $n$.

Taking into account Lemma \ref{L7.6bis} and Cauchy-Schwarz with
respect to the finite measure $Z_r(\o_1,\o) Q^{\o,\kappa_{n}}(\d \omega_1)$,
it is enough to prove that for $r$ a.e.
\begin{equation} \label{E7}
E^{Q^{\o, \kappa_{n}}} \left \{ \(\Phi_{\kappa_{n}}
(X^{\kappa_{n}}(r,Y_r,\o))-
\Phi(X (r,Y_r,\o))\)^2   \varphi(r,Y_r)
Z_r(\cdot, \o) \right \} 
\end{equation}
converges to zero when $n $ goes to infinity.

Since $X^{\kappa}$ constitutes the family of  $\mu$-marginal weighted
 laws of $Y$ under $Q^{\o,\kappa}$, previous expression gives
\begin{align} \label{e7.17bis}
& \int\limits_\R \vert \varphi \vert(r,y)\(\Phi_{\kappa_n}
\(X^{\kappa_n}(r,y,\o)\)-\Phi\(X(r,y,\o)\)\)^2 X^{\kappa_n}(r,y,\o) \d y \nonumber 
\\
\leq & I_{11}(\kappa_n,r)+I_{12}(\kappa_n,r)+I_{13}(\kappa_n,r)+I_{14}(\kappa_n,r),
\end{align}
where we have developed the square in the first line
of \eqref{e7.17bis} using the  definition of 
$\psi$ and $\Phi_\kappa$. Indeed we get 
\begin{align*}
I_{11}(\kappa,r) = &  \int\limits_\R \d y \vert \varphi \vert(r,y) \vert 
\abs{\psi\(X^{\kappa}(r,y,\o)\)-\psi\(X(r,y,\o)\)},\\
I_{12}(\kappa,r) = &  \int\limits_\R \d y 
\vert \varphi \vert(r,y) \vert 
\Phi^2 \(X^{\kappa}(r,y,\o)\) \left \vert  (X-X^{\kappa})(r,y,\o)
 \right \vert,\\
I_{13}(\kappa,r) = &  \int\limits_\R \d y
\vert  \varphi (r,y)\vert  \kappa \abs{X^{\kappa_n}-X}(r,y,\o),\\
I_{14}(\kappa,r) = & \int\limits^T_0 \d r \int\limits_\R \d y \kappa
 \abs{X(r,y,\o)}  \vert \varphi(r,y) \vert.
\end{align*}
We denote $I_{1j}(\kappa):= \int_0^T I_{1j}(\kappa,r) \d r$, $j =1,2,3,4$.
It is of course enough to prove that, up to a subsequence
 $I_{1j}(\kappa_n) \rightarrow 0 $, $j =1,2,3,4$, where $n \rightarrow \infty$.
By Cauchy-Schwarz, $I_{11}^2(\kappa)$ is bounded by
\begin{equation*}
\norm{\varphi}^2_{L^2([0,T]\times\R)} \int\limits^T_0 \d r \int\limits_\R \(\psi \(X^{\kappa}(r,y,\o)\)-\psi\(X(r,y,\o)\)\)^2 \d y.
\end{equation*}
This converges to zero according to Remark \ref{R7.5} 2),
after extracting a  subsequence $\kappa_n)$ (not depending on $\o$). 
The square of the expectation of  $I_{12}(\kappa)$ is bounded by
\begin{equation*}
\norm{\varphi}^2_{L^2([0,T]\times\R)} \int\limits_{[0,T]\times\R}
\d r\d y \Phi^4 \(X(r,y,\o)\)\abs{X^{\kappa}-X}^2(r,y,\o).
\end{equation*}
The expectation of previous expression is indeed uniformly bounded in
 $\kappa$ because of \eqref{e7.13} and \eqref{e7.14bis}. 
So the family of r.v. \\
$\Phi^2 \(X^{\kappa_n}(r,y,\o)\) \left \vert  (X-X^{\kappa})(r,y,\o)\right \vert $ 
is uniformly integrable with respect to the finite
measure $\d P(\o) \vert \varphi\vert(t,y) \d t \d y$.
Consequently $I_{12}(\kappa)$   goes to zero because of \eqref{e7.8}
 in Remark \ref{R7.5} 4).\\
$I^2_{13}(\kappa)$ is bounded by
$\kappa \norm{\varphi}^2_{L^2([0,T]\times\R)}
 \kappa \int\limits_{[0,T]\times\R}\d r\d y \abs{X^{\kappa}-X}^2(r,y,\o).
$
After extracting a subsequence $\kappa_n$, previous expression
 converges to zero because of Lemma \ref{L7.4} c). Finally $I_{14}(\kappa) \substack{\longrightarrow\\n\to \9}0$ by Cauchy-Schwarz and the fact that \\
 $\int\limits_{[0,T]\times\R}\d r\d y X^2 (r,y,\o)<\9 \; \pas$ 
This establishes the proof of Lemma \ref{L7.7}.
\end{proof}
\medskip
Let $(\kappa_n)$ be the sequence introduced by the statement
of Lemma \ref{TStableKolm}.
Previous Corollary \ref{C7.6} and Lemma \ref{L7.7}  have the 
following consequences .
Let  $Q(\d \o_1, \o)$ be the random kernel introduced in Lemma \ref{L7.6}
and the related probability   ${\bf Q}(\d \o_1,\d \o)= \d P(\o) Q(\d \o_1, \o)$.
\begin{cor} \label{C1}
Let $R:\Omega \rightarrow \R$ be a bounded measurable r.v.
Let $\varphi:[0,T] \times \R \rightarrow \R$ be a function 
with compact support.
The sequence 
 \begin{equation} \label{E1bis}
\int_\Omega R(\omega) \d P(\omega) \int_{\tilde \Omega_1} \d Q^{\o,\kappa_n}(\o_1) 
\int_0^T \varphi(r,Y_r) \Phi_{\kappa_n}^2(X^{\kappa_n}(r,Y_r,\omega)) \d r
\end{equation} 
converges, when $n \rightarrow \infty$, to
 \begin{equation} \label{E1ter}
\int_\Omega R(\omega) \d P(\omega) \int_{\tilde \Omega_1} \d Q(\o_1,\o) 
\int_0^T \varphi(r,Y_r) \Phi^2(X(r,Y_r,\omega)) \d r.
\end{equation} 
\end{cor}
\begin{proof} \
We split  the difference between 
\eqref{E1bis} and \eqref{E1ter} which gives
$I_1(n) + I_2(n)$ where 
$$
I_1(n) = \int_\Omega R(\omega) \d P(\omega) \int_{\tilde \Omega_1} 
 \d Q^{\o,\kappa_n}(\o_1)
\left( \int_0^T \varphi(r,Y_r) 
(\Phi_{\kappa_n}^2(X^{\kappa_n}(r,Y_r,\omega)) \d r
- \Phi^2(X(r,Y_r,\omega)) \d r\right),
$$
and
$$
I_2(n) = \int_\Omega R(\omega) \d P(\omega) \int_{\tilde \Omega_1} 
 ( Q^{\o,\kappa_n}(\d \o_1) -  Q(\d \o_1,\o))
\left( \int_0^T \varphi(r,Y_r) 
\Phi^2(X(r,Y_r,\omega)) \d r \right).
$$
We have
$$ \vert I_1(n) \vert \le  
2 \Vert \Phi \Vert_\infty \Vert R \Vert_\infty
 \int_\Omega \d P(\omega) 
\int_{\tilde \Omega_1} 
 \d Q^{\o,\kappa_n}(\o_1)
\left( \int_0^T \vert \varphi(r,Y_r)\vert  
\vert \Phi_{\kappa_n}(X^{\kappa_n}(r,Y_r,\omega)) \d r
- \Phi(X(r,Y_r,\omega))\vert \d r\right).
$$
$I_1(n)$ converges to zero by Lemma \ref{L7.7}.
Concerning $I_2(n)$, by Fubini's theorem,
we first observe that
$$ I_2(n) =  \int_0^T \d r  \int_\Omega \d P(\omega)   
    \left( \int_{\tilde \Omega_1} 
 ( Q^{\o,\kappa_n}(\d \o_1) -  Q(\d \o_1,\o))
 \varphi(r,Y_r) 
\Phi^2(X(r,Y_r,\omega)) R(\o) \right).$$
We apply now Corollary \ref{C7.6}, setting  
for fixed $r$, 
 $F(\o_1,\o) =  R(\o) \varphi(r,\o_1(r))\Phi^2(X(r,\omega_1(r),\omega))
$
 and the result follows.

\medskip


\end{proof}

 \begin{description}
 \item{5)}
 We go on 
 with the proof of Theorem \ref{T73}. 

 We want now to prove that $Y(\cdot,\omega)$ is a (weak)  solution of  
  \begin{equation} \label{e7.15}
  Y_t = Y_0 + \int\limits^t_0 \Phi \(X(s,Y_s,\cdot)\)\d \beta^\omega_s,
  \end{equation}
for some Brownian motion  $\beta^\omega $. This is related to item 1) of Definition \ref{DWeakStrong}
 with $\gamma(t,\xi,\omega) = \Phi(X(t,\xi,\omega))$.
According to
  Remark \ref{R2.10} c), for this
it is
 enough to show that for 
$\pas \o$
 $Y(\;\cdot\;,\o)$  
is a solution of the following (local) martingale problem. For every $f\in C^{1,2}([0,T]\times \R)$ with compact support, the process
$$ Z_t^f:=f(t,Y_t)-f(0,Y_0)-\frac{1}{2} \int\limits^t_0 
 \partial^2_{xx}f(r,Y_r) \Phi^2\(X(r,Y,\o)\)\d r
- \int_0^t \partial_r f(r,Y_r) dr, 
$$
is a (local) martingale  under $Q^\o$.
\end{description}

This will be a consequence of the lemma below.
\begin{lemma} \label{L78}
Let $F$ be a bounded $\shf_s$-measurable,
let $\sha:C([0,s]) \rightarrow \R$
bounded continuous functional. Let $G = \sha(Y_r, r\le s)$.
Then, for $0 \le s \le t \le T$ we have
\begin{equation}\label{E78bis}
E(F E^{Q^{\o}}(G Z_t^f)) =  E(F E^{Q^{\o}}(G Z_s^f)).
\end{equation}

\end{lemma}
\begin{proof} \
We set 
$$ Z_t^{\kappa,f} = f(t,Y_t)-f(0,Y_0)-\frac{1}{2} \int\limits^t_0 
 \partial^2_{xx}f(r,Y_r) \Phi_\kappa^2\(X^\kappa(r,Y,\o)\)\d r
- \int_0^t \partial_r f(r,Y_r) \d r.
$$
Let $(\kappa_n)$ be the sequence introduced by Lemma \ref{TStableKolm}.
The difference of the right and left-hand side of \eqref{E78bis}
is the sum $(I_1 + I_2 + I_3)(\kappa_n)$
where 
\begin{eqnarray*}
I_1(\kappa) &=&  E\left(F (E^{Q^{\o}}(G Z_t^f) - 
E^{Q^{\o,\kappa}}(G Z_t^{\kappa, f}))\right) \\
I_2(\kappa) &=&  E\left(F E^{Q^{\o, \kappa}}
(G (Z_t^{\kappa,f} - Z_s^{\kappa,f}))\right)  \\
I_3(\kappa) &=&  E\left(F (E^{Q^{\o, \kappa}}(G Z_s^{\kappa, f}) - 
E^{Q^{\o}}(G Z_s^f))\right).
\end{eqnarray*}
$I_1(\kappa_n) + I_3(\kappa_n)$ converges to zero 
by Lemma \ref{TStableKolm}, Corollary \ref{C1} and Lemma \ref{L7.7}.
$I_2(\kappa_n) = 0$ since $Z^{\kappa,f}$ is a $Q^{\kappa, \o}$-martingale.
\end{proof}


\begin{description}
\item{6)} After previous intermediary step 5) we need to show that $Y$ defined in \eqref{ENummer}
is a weak-strong solution of DSDE$(\gamma,x_0)$ with
$\gamma(s,\xi,\omega) = \Phi(X(s,\xi,\omega))$ and $X$ is a solution of \eqref{PME}.
We recall that the kernel $Q(\cdot, \o)$ has been introduced through
 Lemma \ref{TStableKolm}
on $(\tilde \Omega_1 \times \Omega, \shh \otimes \shf)$. 
 So, according to step 5), under  $Q^\omega:=Q(\cdot,\omega)$,
 $Y$ is a martingale with
$[Y]_t =\int\limits^t_0 \Phi^2\(X(s,Y_s,\o)\)\d s$. To conclude the proof
of item 1) in Definition \ref{DWeakStrong}, it remains to construct the 
suitable required process $B$. 
For this, we need
to enlarge the probability space $\tilde \Omega_1$ as follows.
We set $\Omega_1 = \tilde \Omega_1 \times C([0,T];\R)$; the second component allows to
define a Brownian motion.
By an abuse of notation, we  set again 
$ Y_t(\o_1,\o)= \omega_1^0(t),$
 this time with 
$\omega_1 = (\omega_1^0,\omega_1^1).$
In spite of adding the component $\omega^1_1$, in step 5) 
  we have already shown $Q^{\omega} := Q(\;\cdot\;,\o)$, is by construction
  the law of
 $Y(\cdot,\o)$.
We need to construct a process $B$ on $\Omega \times \Omega_1$,
such that for almost all $\omega$, $B(\cdot,\omega)$ is a 
 $Q^\omega$-Brownian motion and
\eqref{DWS} holds for $\gamma(t,\cdot, \omega) = \Phi(X(t,\cdot,\omega))$.

On $\Omega_1 $ we set $\beta_t(\omega_1)  =   \omega_1^1(t)$. We equip 
$C([0,T];\R)$ in $\Omega_1$ with the Wiener measure $\shw$ so that 
$\beta$ is a standard Brownian motion on $\Omega_1$. 
$\beta$ can also be considered to be a Brownian motion on 
$\Omega_0 = \Omega_1 \times \Omega$ which
is $Q^\o$-independent of $Y$ for $P$-almost all  $\omega \in \Omega$.
Of course $\beta$ is also independent of $Y$ on
the probability space $(\Omega_1\times \Omega, \shb(\Omega_1) \times \shf, 
\d {\bf Q}(\omega_1,\omega):= Q^\omega(\d \omega_1) \d P(\omega))$.
$\beta$ is also independent of $(\shf_t)$.

We set now
$$ B_t(\cdot,\omega) = \int_0^t \d Y_s(\cdot,\omega) 
1_{\{\gamma(s,\xi,\omega) \neq 0\}} \frac{1}{\gamma(s,\xi,\omega)}
+ \int_0^t   1_{\{\gamma(s,\xi,\omega) = 0\}} \d \beta_s.$$ 
 Now for $Q^\omega$-a.s. the quadratic variation of 
the $Q^\omega$-martingale $B(\cdot,\omega)$
is $t$, so that, by L\'evy characterization theorem,
 $B(\cdot, \omega)$ is a Brownian motion under $Q^\omega$.

It remains to show  items 2) and 3) of the definition  of  weak-strong solution.
Let $(\shy_t)$ be the canonical filtration of the process 
$Y(\cdot, \omega)$. Item 3) follows because of item 1)
and because $\gamma(t,\cdot, \omega) = \Phi(X(t,\cdot,\omega))$
is progressively measurable.
Concerning item 2) we see that 
under ${\bf Q} $ defined by $P$ and the kernel $Q(\cdot,\omega)$,
$W^1, \ldots,W^N$ are ${\bf Q}$-martingales with $(\shg_t)$
as defined in Definition \ref{DWeakStrong}.
Indeed let $F$ be a bounded $\shf_s$-measurable random variable and $G$ be a
bounded $\shy_s$-measurable r.v. 
Let $1 \le i \le N$. 
By item 3) $E^{Q^{\omega}} (G)$ is $\shf_s$-measurable, so 
$$ E^{\bf Q}((W^i_t - W^i_s) F G) =  E((W^i_t - W^i_s) F 
E^{Q^{\omega}} (G) ) =  0, $$
since $W^i$ is an  $\shf_s$-martingale.

\item{7)}
 The final step consists in proving that $X$ is the family of $\mu$-marginal
weighted laws of $Y$.
We need  to show that for almost all $\o$, for every $t\in[0,T],\; \varphi\in \shs(\R)$, that
\begin{equation}\label{e7.21}
\int\limits_\R \d \xi\varphi(\xi) X(t,\xi,\o) = E^{Q^\o} \(\varphi(Y_t)\she_t
 \(\int\limits^\cdot_0 \mu (\d s,Y_s)\)\).
\end{equation}
Since both sides of previous equality are $\shf_t$-measurable,
given a bounded $\shf_t$-measurable random variable $R$ it will be enough
to show that
\begin{equation}\label{e7.21bis}
 \int_\O \d P(\o)R(\o) \int\limits_\R \d \xi\varphi(\xi) X(t,\xi,\o) = 
\int_\O \d P(\o)  R(\o)
E^{Q^\o} \(\varphi(Y_t)\she_t
 \(\int\limits^\cdot_0 \mu (\d s,Y_s)\)\).
\end{equation}
 Let $\o\in\O$ outside some $P$-null set.\\
By step 1) of the proof of this Theorem \ref{T73},
we know that $X^\kappa$ fulfills, for almost all $\o$,
$$
\int\limits_\R \d \xi X^\kappa (t,\xi)\varphi(\xi)
=  E^{Q^\kappa (\;\cdot\;,\o)} \(\varphi(Y_t)\she_t \(\int\limits^\cdot_0
 \mu (\d s,Y_s)\)\),
$$
for all $\varphi\in\shs(\R)$.
Consequently if $({\kappa_n})$ is the sequence obtained via Lemma 
\ref{L7.6}, we have
\begin{equation}\label{e7.20}
\int_\O \d P(\o)R(\o)  \int\limits_\R \d \xi X^{\kappa_n} (t,\xi)\varphi(\xi)
=  \int_\O \d P(\o)R(\o)  E^{Q^{\o, \kappa_n}}
 \(\varphi(Y_t)\she_t \(\int\limits^\cdot_0 \mu (\d s,Y)\)\),
\end{equation}
for every $\varphi\in\shs(\R)$.


Since $t \mapsto X(t,\;\cdot\;)$ is continuous from $[0,T]$ to 
$\shs'(\R)$ and the right-hand side of \eqref{e7.21bis} 
 is continuous on $[0,T]$ for fixed $\varphi\in\shs(\R)$, it is enough to show 
\eqref{e7.21bis} for almost all $t\in[0,T]$.\\
Now for almost all $t$, the left-hand side of \eqref{e7.21bis} is approached by the left-hand side of \eqref{e7.20}.
Let us fix $t\in [0,T]$.
 It remains to show that the right-hand side of \eqref{e7.21bis} is the limit of the right-hand side of \eqref{e7.20}.
 We fix $\o \in \Omega $ outside a null set.
We set  $ \she_t := \she_t
 \(\int\limits^\cdot_0 \mu (\d s,Y_s)\), \ t \in [0,T].$
By Theorem 2 of \cite{RVSem} and uniform integrability
arguments, similarly as after \eqref{EMargU},
we have 
$$
\she_t = \exp (\psi_\omega(Y)), $$
 where $\psi_\omega: \tilde \Omega_1 \rightarrow \R$
is a continuous modification of 
$$\o \mapsto \(\eta \mapsto \int\limits^t_0 e^i (\eta_s) \d W^i_s -\frac{1}{2} \int\limits^t_0 
e^i (\eta_s)^2 \d s)\).$$ Indeed, previous random field, indexed by
 $\eta \in   \tilde \O_1,$ admits a continuous modification; to prove this
we make use of Kolmogorov-Centsov theorem and  Doob's inequality, which says that 
 for any $0 \le i \le N$, there is a constant ${\rm const}= {\rm const}((e^i)'$)
with
$$ E\left( \vert \int \limits^t_0 (e^i (\eta^1_s) - e^i (\eta^2_s)) \d W^i_s \vert^4\right)
 \le {\rm const} \sup_{s \in [0,T]} \vert \eta^1 - \eta^2\vert^2(s), 
\ \eta^1, \eta^2  \in \tilde \O_1.  $$


 At this point we fix $M > 0$.
We decompose the difference of the right-hand sides of \eqref{e7.20} and \eqref{e7.21bis} as
\begin{equation} \label{J123} 
J_1(n, M) + J_2(n,M) + J_3(n,M),
\end{equation}
 where 
 $$J_1(n,M) =  \int_\O \d P(\o)R(\o)  E^{Q^{\o, \kappa_n}} 
  \left(\varphi(Y_t)\she_t  - \varphi(Y_t) (\she_t \wedge M)
   \right),$$
 $$J_2(n,M) =  \int_\O \d P(\o)R(\o)  (E^{Q^{\o, \kappa_n}} - E^{Q^\o}) 
  \left(\varphi(Y_t)  (\she_t \wedge M) \right),$$
 $$J_3(n,M) = \int_\O \d P(\o)R(\o) E^{Q^\o} \left( \varphi(Y_t) 
 (\she_t \wedge M) -  \varphi(Y_t)\she_t \right).$$
Setting ${\bf Q}^{\kappa_n}(\d \o,\d \o_1) = \d P(\o) Q^{\o,\kappa_n}(\d \o_1)$,
by Cauchy-Schwarz and Chebyshev inequalities, for every $p >1$, we have
$$  \vert J_1(n,M) \vert = \left \vert \int_{\O_1 \times \O} \d Q^{\kappa_n}   
  \varphi(Y_t)\she_t 1_{\{\she_t > M\}} \right \vert \le 
 \Vert \varphi \Vert_\infty  \frac{E^{Q^{\kappa_n}}( \she_t^p)}{M^{p-1}}.$$
By Lemma \ref{L7.6bis}, we get
$\sup_n \vert J_1(n,M)\vert \rightarrow 0$ if $M \rightarrow \infty$. 
By a similar reasoning, replacing 
 $Q^{\kappa_n}(\d \o,\d \o_1)$
with  $Q(\d \o,\d \o_1) = \d P(\o) Q^{\o}(\d \o_1)$,
we can prove that 
$\sup_n \vert J_3(n,M) \vert \rightarrow 0$. 
Let $\varepsilon > 0$. Let $M$ such that
$\sup_n \vert J_1(n,M) + J_3(n,M) \vert \le \varepsilon.$
On the other hand we have 
$$J_2(n,M) =  \int_\O \d P(\o)R(\o)  \left(E^{Q^{\o, \kappa_n}} - E^{Q^\o}\right) 
  (\varphi(Y_t)(\psi_\o(Y) \wedge M)).$$
Since for almost all $\o$, $F(\eta,\o) :=  R(\o) \varphi(\eta(t)) \psi_\o(\eta)$
is bounded and continuous, Corollary \ref{C7.6} implies that $J_2(n,M)$
 goes to zero when $n \rightarrow \infty$.

Taking the limsup in \eqref{J123} we get
$$ \limsup_{n \rightarrow \infty} \vert  J_1(n,M) + J_2(n,M) + J_3(n,M) \vert \le \varepsilon.$$
Since $\varepsilon$ is arbitrarily small, we get
$ \lim_{n \rightarrow \infty}  \vert  J_1(n,M) + J_2(n,M)+ J_3(n,M) \vert = 0$
and the result follows.



\end{description}
\end{prooff}

\begin{appendix}
\section{Technicalities}

\setcounter{equation}{0}

\begin{prop} \label{A4}
Let $Y_0$ be distributed according to $x_0$. Let $a:[0,T]\times\R\to\R$ be a Borel function such there are $0<c<C$ with
$ c \leq a(s,\xi)\leq C, \quad \forall \; (s,\xi)\in[0,T]\times \R.$
We fix $0 \le r  \le t \le T$.
We set $a_n(t,x)=\int\limits_\R \rho_n (x-y) a(t,y)\d y$ where $(\rho_n)$
 is the usual sequence of mollifiers converging to the Dirac delta.
The unique solutions $S^n$ to
$S^n_t=Y_0 + \int^t_r a_n (s, S^n_s)\d B_s,$
$B$ being a classical Wiener process, converges in law to the (weak unique solution) of
$ S_t=Y_0+\int^t_r a (s, S_s)\d B_s. $
\end{prop}
\begin{proof}
The proof follows by standard arguments, see Stroock-Varadhan (\cite{sv}, Problem 7.3.3), tightness
and Kolmogorov-Centsov type arguments. For a detailed proof, the reader may consult \cite{Bar-Roc-Rus-2200}.
\end{proof}

{\bf Proof} (of Corollary \ref{C7.6}).

By \eqref{TE3}, the family  $(Q^{\kappa_\ell}, \ell \in \N, \o \in \O)$
is tight. So, for every positive integer $n$ there exists a compact subset 
$K_n$ of $\tilde \O_1$ such that
\begin{equation} \label{EC1}
Q^{\kappa_\ell}(K_n^c,\o) < \frac{1}{n},  \forall  \ell \in \N, \omega \in \O.
\end{equation}
Since each $C(K_n):=C(K_n;\R)$ is separable with respect to the
sup-norm $\Vert \cdot \Vert_\infty$ then $C(K_n),\Vert \cdot \Vert_\infty$ is a 
separable Banach space. So we apply  Appendix 1, Lemma A.1.4 in \cite{prevot},
to the map $\O \ni \o \mapsto F(\cdot \vert_{K_n},\o) \in C(K_n)$,
where $ F(\cdot \vert_{K_n}, \o)$ denotes the map 
$K_n \ni \eta \mapsto F(\eta, \o)$. Therefore we can find
a sequence $\tilde F_{n,k}: \O \rightarrow C(K_n)$, $\o \mapsto 
\tilde F_{n,k}(\cdot,\o) \in K_n$ such that for
$ \Vert F \Vert_\infty:= \sup_{\eta \in \tilde \O_1, \o \in \O} 
\vert F(\eta,\o)\vert,$
we have 
$$ \Vert \tilde F_{n,k} \Vert_\infty \le 1 + \Vert F \Vert_\infty, 
\quad \tilde F_{n,k}(\O) \subset \{\tilde g^{(1)}_{n,k}, \ldots, 
\tilde g^{(N_{n,k})}_{n,k}\} \subset C(K_n),$$
 where $\tilde g^{(i)}_{n,k} \neq \tilde g^{(j)}_{n,k}$ if $ i \neq j$,
and for all $\o \in \O$
\begin{equation} \label{EC2}
\sup_{\eta \in K_n} \vert F(\eta, \o) - F_{n,k}(\eta, \o)\vert \rightarrow 0,
\end{equation}
as $k \rightarrow \infty$. Clearly, for all $\o \in \O, \quad
 \tilde F_{n,k}(\cdot, \o) = \sum_{j=1}^{N_{n,k}} \tilde g^{(j)}_{n,k} 
1_{\{\tilde g^{(j)}_{n,k}\}} \circ \tilde F_{n,k}(\cdot, \o). $
By Tietze's extension theorem there exist  extensions
 $g^{(1)}_{n,k}, \ldots, 
 g^{(N_{n,k})}_{n,k} \in C(\tilde \O_1)$ 
of $\tilde g^{(1)}_{n,k}, \ldots, 
 \tilde g^{(N_{n,k})}_{n,k}$
such that for all  
$1 \le j \le N_{n,k}$, 
$ \sup_{\eta \in \tilde \O_1}  \vert g^{(j)}_{n,k}(\eta) \vert 
\le \sup_{\eta \in \tilde K_n}  \vert \tilde g^{(j)}_{n,k}(\eta)\vert.$
Now we define $F_{n,k}: \Omega  \rightarrow C(\tilde \O_1),$
$ \o \mapsto F_{n,k}(\cdot,\o)$ by 
 $$  F_{n,k}(\cdot, \o) = \sum_{j=1}^{N_{n,k}} g^{(j)}_{n,k} 
1_{\{\tilde g^{(j)}_{n,k}\}} \circ \tilde F_{n,k}(\cdot, \o). $$
Clearly, still
\begin{equation} \label{EC3}
\Vert F_{n,k} \Vert_\infty \le 1 + \Vert F \Vert_\infty.
\end{equation}
Note that for all $\eta \in \tilde \O_1$
$$ \tilde F_{n,k}(\eta, \o) = \sum_{j=1}^{N_{n,k}} g^{(j)}_{n,k}(\eta) 
1(\eta)_{\{\tilde g^{(j)}_{n,k}\}} \circ \tilde F_{n,k}(\eta, \o), $$
hence of the form that Lemma \ref{TStableKolm} applies.
Therefore using the standard notation
$ \mu(f):= \int f \d \mu,$
for a measure $\mu$ and a function $f$, 
we can argue as follows. 
Fix $n \in \N$. Then for all $\ell, k \in \N$
\begin{eqnarray*}
&& \left \vert \int Q^{\kappa_\ell}(F(\cdot,\o), \o) P(\d \o)  -
 \int Q(F(\cdot,\o), \o)P(\d \o) \right \vert \\
&\le &
\left \vert \int Q^{\kappa_\ell}(F(\cdot,\o) 1_{K_n}, \o)P(\d \o)  - 
\int Q(F(\cdot,\o)1_{K_n}, \o)P(\d \o) \right \vert 
 + \frac{2}{n} \Vert F \Vert_\infty \\ 
 & \le  & 
 \int  \underbrace{Q^{\kappa_\ell}(\left \vert F(\cdot,\o) -  
F_{n,k}(\cdot,\o) \right \vert 1_{K_n},\o)}_{\le \sup_{\eta \in K_n} 
\vert F(\eta,\o) - F_{n,k}(\eta,\o)\vert}     P(\d \o) \\
&+& \left \vert \int Q^{\kappa_\ell}(F_{n,k}(\cdot,\o), \o)P(\d \o) -
       \int Q(F_{n,k}(\cdot,\o), \o)P(\d \o) \right \vert \\
&+& \frac{2}{n} (1 + \Vert F \Vert_\infty) + 
 \int  Q(\vert F(\cdot,\o) - F_{n,k}(\cdot,\o)\vert, \o) 1_{K_n}
   P(\d \o) +  
\frac{2}{n} \Vert F \Vert_\infty.
\end{eqnarray*}
The first inequality is a consequence of \eqref{EC1}, the second one of
 \eqref{EC1} and  \eqref{EC3}.
Now, letting first $\ell \rightarrow \infty$ (using Lemma \ref{TStableKolm}),
then $k \rightarrow \infty$ (using \eqref{EC2}) and finally
$n \rightarrow \infty$, the assertion follows.
\qed

\section{Uniqueness for the porous media equation 
with  noise}

\label{SUniqueness}

\setcounter{equation}{0}

We state here a general uniqueness lemma which only holds 
under  even weaker hypotheses  than Assumption \ref{E1.0} i.e.
$\psi: \R \rightarrow \R$ is Lipschitz and that the functions
belong to $W^{1,\infty}$.

\begin{theo} \label{TB1}  Let $x_0 \in \shs'(\R^d)$ and
suppose $\psi:\R \rightarrow \R$ to be Lipschitz.
Then equation \eqref{PME} admits at most one
 solution among the random fields $X:]0,T]\times\R\times\O\to\R$ such that 
\begin{equation}\label{eFB1}
\tag{B.1}
\int_{[0,T]\times\R} X^2(s,\xi)\d s\d \xi <\9 \quad \text{a.s.}
\end{equation}
\end{theo}
\begin{rem}\label{remB2bis} 
\begin{enumerate}
\item  
Suppose moreover that
$ e^i, 0 \le i \le N,  \ {\rm belong \ to} \ H^1.$
 If $x_0 \in L^2$ or $\psi$   is non-degenerate
then  Theorem 3.4 of \cite{BRR3} provides an existence theorem for \eqref{PME}.
It states the existence of a random field $X$
such that 
\begin{equation}\label{eFB1bis}
E \left(\int_{[0,T]\times\R} X^2(s,\xi)\d s\d \xi\right) <\9,
\end{equation}
such that 
$t \mapsto X(t,\cdot)$ belongs to  $C([0,T]; H^{-1}(\R))$ and
 $t \mapsto  \int_0^t \psi(X(s,\cdot)) \d s \in  C([0,T]; H^1(\R))$ a.s.
\item So, under the assumption of item 1., the solution $X$
is unique among those fulfilling \eqref{eFB1}.
\item $X$ of point ii) fulfills the equation, for almost all $\o$,
in $H^{-1}$
 \begin{equation}\label{equater}
X(t,\;\cdot\;)=x_0+ \int\limits^t_0 \Delta(\psi (X(s,\;\cdot\;)))\d s 
 + \int\limits^t_0 \mu ( \d s , \cdot) X(s,\cdot), \quad t\in [0,T].
\end{equation}
\end{enumerate}


\end{rem}

The proof of Theorem \ref{TB1} is a consequence of the result stated 
 in Theorem \ref{TB1} of \cite{BRR4},
see also \cite{Bar-Roc-Rus-2200}.

\end{appendix}
\bigskip

{\bf ACKNOWLEDGMENTS} 
\noindent

The authors are grateful to the associated Editor and the two 
Referees for their careful reading of the first versions of
the manuscript which has allowed them to considerably improve 
the quality of the paper.
Financial support through the SFB 701 at Bielefeld University and
NSF-Grant 0606615
is gratefully acknowledged. 
The third named author  benefited partially from the support of the 
``FMJH Program Gaspard Monge in optimization and operation research'' 
(Project 2014-1607H) and from 
 the ANR Project MASTERIE 2010 BLAN 0121 01.
Part of this work was written during a stay of the first
and third named authors at
the Bernoulli Center (EPFL Lausanne).

\addcontentsline{toc}{section}{References}
\bibliographystyle{plain}
\bibliography{BRR_Bibliography}

\end{document}